\documentclass[reqno,11pt]{amsart}
\usepackage{amsmath,amsthm,amsfonts,color,graphicx,bbm,mathabx}

\usepackage[latin1]{inputenc}
\usepackage[makeroom]{cancel}
\usepackage{pdfsync,subcaption,multirow,float,esint}
\usepackage{filecontents,diagbox,bbold,dsfont,enumerate}

\definecolor{lightGray}{RGB}{235,235,235}
\definecolor{orange}{RGB}{255,128,0}
\definecolor{ucib}{RGB}{0,36,105}
\definecolor{mygreen}{RGB}{0,128,0}
\definecolor{lightBlue}{RGB}{102,153,204}

\newtheorem{thm}{Theorem}[section]
\newtheorem{lem}[thm]{Lemma}
\newtheorem{cor}[thm]{Corollary}
\newtheorem{prop}[thm]{Proposition}

\newtheorem{rems}[thm]{Remarks}
\newtheorem{rem}[thm]{Remark}

\newtheorem{deff}[thm]{Definition}

\DeclareMathAlphabet{\mathpzc}{OT1}{pzc}{m}{it}

\numberwithin{equation}{section}

\begin{document}
\bibliographystyle{plain}

\title{Geometric Kernel Interpolation and Regression}

\author{Patrick Guidotti}
\address{University of California, Irvine\\
Department of Mathematics\\
340 Rowland Hall\\
Irvine, CA 92697-3875\\ USA }
\email{gpatrick@math.uci.edu}

\begin{abstract}
  Exploiting the variational interpretation of kernel interpolation we
  exhibit a direct connection between interpolation and regression,
  where interpolation appears as a limiting case of regression. By
  applying this framework to point clouds or samples of smooth
  manifolds (hypersurfaces, in particular), we show how fundamental
  geometric quantities such as tangent plane and principal curvatures
  can be computed numerically using a kernel based (approximate) level
  set function (often a defining function) for smooth hypersurfaces.
  In the case of point clouds, the approach generates an interpolated
  hypersurface, which is an approximation of the underlying manifold
  when the cloud is a sample of it. It is shown how the geometric
  quantities obtained can be used in the numerical
  approximation/computation of geometric operators like the surface
  gradient or the Laplace-Beltrami operator in the spirit of kernel
  based meshfree methods. Kernel based interpolation can be extremely
  ill-posed, especially when using smooth kernels, and the regression
  approximation offers a natural regularization that proves also quite
  useful when dealing with geometric or functional data that are
  affected by errors or noise.
\end{abstract}

\keywords{Kernel based interpolation and regression, point cloud
  interpolation and analysis, functions and operators on point clouds}
\subjclass[1991]{}

\maketitle

\section{Introduction}
Kernel Methods play an important role in interpolation theory as they
nicely bridge the gap between infinite dimensional function spaces and
finite dimensional approximations thereof. We shall make this statement more
precise below when we highlight the optimization aspect of kernel
interpolation. Kernel methods are meshfree methods that come with a
solid theoretical foundation \cite{Fass15,W04}. They can be used
for interpolation as well as for a variety of other purposes that
range from the numerical resolution of differential equations and
(initial) boundary value problems to classification and statistical
learning \cite{RW06}. A particularly relevant connection to Gaussian Process
Regression will be briefly discussed as it provides an alternative 
justification for the regularized kernel-based ``quasi''-interpolation
that proves useful in the context of noisy data.

The main contribution of this paper consist of the realization, new to
the best of our knowledge, and subsequent exploitation, of the fact
that kernel methods (and the related regression methods) can be used
directly to construct viable ``defining'' (level set) functions of hypersurfaces (or
more general submanifolds), which then give direct access to global
(localizable) interpolation of their geometric properties (normals and
curvatures) starting with a (not necessarily uniform) sample of their
points (the point cloud). With these in hand, 
it also becomes possible to stably compute geometric operators of functions
simply given through their values on the unstructured
(sampling) point cloud in the spirit of meshfree methods. The point cloud 
(sample) is viewed as a discrete embedding and the proposed use of
kernel methods combined with formul\ae{} in terms of ambient space
coordinates yields numerical discretizations in extrinsic
coordinates. In other words, geometric quantities and operators can be
computed directly from the point cloud and applied to functions
defined on it. The approximate computation of curvatures and
differential operators on manifolds is immediately relevant within
applied mathematics and numerical analysis as they can directly be
employed in order to numerically solve a variety of natural PDEs on
manifolds. It does, however, also play an important role in computer
graphics and vision (\cite{K92,Taubin95,MDSB03}). The use of geometric
operators can also be used in the case of noisy or corrupted samples.
Following ideas first pioneered in a full space context, denoising of
(surface) meshes  can, for instance, be performed by using normal
motion by mean curvature (see \cite{DMSB99}). The approach to
corrupted or noisy clouds taken here does not require explicitly
evolving the noisy sample since denoisining is implicit and automatic
in the sense that the regularization yielding approximate
interpolation or regression results in smoother level sets
of the approximate interpolant. Numerical access to geometric
information is also of interest in applications beyond
denoising. Examples of such applications are mesh simplification
(\cite{GH97,HG99}), surface modeling \cite{MS92}, and fairing of
meshes \cite{CDR03}, where a formulation as a Laplace-Beltrami
diffusion of the coordinates is used. Other applications include
surface reconstruction \cite{BTSAGSS17},  shape recognition or registration
\cite{GG06}, cortical morphology and organ shape analysis (see
e.g. \cite{FD00,DH22,SZ25}), and many 
more. It is worth mentioning that, contrary to what is typically done
in the literature \cite{CBCMDMBE01,W04,Fass15,Pi12}, we work with the
point cloud itself without augmenting it with two layers of additional points
placed at a fixed distance from the surface in normal and negative normal
direction (which requires additional knowledge about the normal
vectors and introduces additional parameters that influence the
accuracy of the method). Finally the approach is meshfree and, as
such, does not require a sophisticated structured representation of
the surfaces of interest and, in fact, in a smooth 
context, works also with non-necessarily dense or regular
discretizations/samples as will be shown in the final section.

The analysis proposed in this paper offers some new insights but also
integrates many previous observations that have appeared in the kernel
interpolation and radial basis function literature (including
applications to PDEs), in applications to visualization for the
reconstruction, representation, and processing of curves and surfaces,
and in the statistical literature.

It was already observed in \cite{FW12} that positive definite full
space kernels $K: \mathbb{R}^d\times \mathbb{R}^d\to \mathbb{R}$, 
can be used to perform interpolation of functions defined on embedded
manifolds $\mathcal{M}\subset \mathbb{R}^d$, not necessarily of
codimension one. In \cite{Pi12} it is shown how radial basis functions can
be successfully employed in order to compute numerical solution to PDEs
on surfaces. The use of kernel interpolation for scattered data is
also discussed in \cite{W04}, where one also finds an exhaustive
treatment of positive kernels, kernel interpolation theory, approximation
results, and much more. Finally, the spline based regression
introduced in \cite{Wah90} also exploits the structure and framework
provided by general kernel methods even if in a specific
incarnation and was used in \cite{CBCMDMBE01} for dealing with noisy
surface data in a similar way as it is done here in a more general
context. We shall give additional detail in the following
sections. Finally there is also a natural connection to manifold
learning but we refer to \cite{G25} for a more thorough discussion of
this aspect.

\section{Kernel Interpolation}
While we refer the reader to the literature \cite{W05,Fass15} for a
comprehensive account of kernel methods, we provide a brief
introduction here with the goal of revealing the insight that leads to
the geometric use announced in the introduction.
\begin{deff}
An at least continuous kernel $K:\mathbb{R}^d\to \mathbb{R}$ is called
{\em positive definite} if and only if the matrix
$$
K(\mathbb{X},\mathbb{X})=[K(x_j-x_k)]_{1\leq j,k\leq |\mathbb{X}|}\in
\mathbb{R}^{|\mathbb{X}|\times|\mathbb{X}|}
$$
is symmetric positive definite for every choice of a set
$\mathbb{X}=\{ x^1,\dots,x^{|\mathbb{X}|}\}$ containing
$|\mathbb{X}|<\infty$ distinct points.
\end{deff}
This allows one to use $\big\{ K(\cdot-x_k \,\big |\,
k=1,\dots,|\mathbb{X}|\big\}$ as ``basis functions'' in the
interpolation problem consisting of finding a function
$f:\mathbb{R}^d\to \mathbb{R}$ with $f(x^j)=y^j\in \mathbb{R}$ for
$j=1,\dots,|\mathbb{X}|$, or $f(\mathbb{X})=\mathbb{Y}$ for
short. Indeed the corresponding system takes the form
$$
\bigl(\sum_{k=1}^{|\mathbb{X}|}K(\cdot-x^k)\lambda_k\bigr)\big |_\mathbb{X}
=K(\mathbb{X},\mathbb{X})\Lambda=\mathbb{Y}
$$
and is solvable whenever $K$ is positive definite. In this way, one
obtains an everywhere defined interpolant
\begin{equation}\label{interpolant}
  u_{\mathbb{X},\mathbb{Y}}=K(\cdot,\mathbb{X})K(\mathbb{X},\mathbb{X})^{-1}\mathbb{Y}.
\end{equation}
The usefulness and power of the method stems from the fact that it is
possible to associate to $K$ an infinite dimensional Hilbert space of functions
$\mathcal{H}_K$, the so-called reproducing kernel Hilbert space, often
known as the native space, with the remarkable property that
\begin{equation}\label{op}
  u_{\mathbb{X},\mathbb{Y}}=\operatorname{argmin}_{u(\mathbb{X})=\mathbb{Y}}E_0(u),\:
  E_0(u)=\frac{1}{2}\|u\|^2_{\mathcal{H}_K}. 
\end{equation}
The Hilbert space $\mathcal{H}_K$ is the completion of the finite
linear combinations
$K(\cdot,\mathbb{X})\alpha=\sum_{k=1}^{|\mathbb{X}|}\alpha_kK(\cdot-x^k)$
with norm defined by
$$
\bigl( K(\cdot,\mathbb{X}_1)\alpha^1\big |
K(\cdot,\mathbb{X}_2)\alpha^2\bigr)
=(\tilde \alpha ^1)^\top K(\mathbb{X},\mathbb{X})\tilde \alpha^2,
$$
where $\mathbb{X}=\mathbb{X}_1\cup \mathbb{X}_2$ and $\tilde \alpha_i\in
\mathbb{R}^{|\mathbb{X}|}$ is the trivial extension of the vector
$\alpha_i\in \mathbb{R}^{|\mathbb{X}_i|}$ obtained by setting to zero
the missing components corresponding to $\mathbb{X}_i \setminus
\mathbb{X}_j$ ($j\neq i$) for $i=1,2$.
\begin{rems}
  \begin{enumerate}[(a)]
  \item The fact that the minimizer has the form
    $K(\cdot,\mathbb{X})\Lambda$ for some $\Lambda\in
    \mathbb{R}^{|\mathbb{X}|}$ goes by the name {\em Representer
      Theorem} in the more applied literature.
  \item The name Reproducing Kernel Hilbert Space (RKHS) is justified
    by the central property that
    $$
    \int K(x-y)u(y)\, dy=u(x),\: x\in \mathbb{R}^d,
    $$
    for every $u\in \mathcal{H}_K$. If $u$ is replaced by a translate
    of the kernel, the integration yields back the kernel itself. This
    formula also shows that $\delta_x\in \mathcal{H}_K^*$.
  \end{enumerate}
\end{rems}
In order to introduce the proposed extension, we now turn things on
their head and start with the optimization problem \eqref{op} for
which it is natural to consider the closely related problem
\begin{equation}\label{opa}
  \operatorname{argmin}_{u\in \mathcal{H}_K}\frac{1}{2}\left\{\|
      u\|^2_{\mathcal{H}_K}+ 
      \frac{1}{\alpha}|u(\mathbb{X})-\mathbb{Y}|^2\right\}
    =\operatorname{argmin}_{u\in \mathcal{H}_K}E_\alpha(u)
\end{equation}
for $\alpha>0$. One can think of Problem \eqref{op} as the pure
interpolation problem and of \eqref{opa} as the approximate
interpolation or regression problem. For these problems we have the
following basic results.
\begin{prop}\label{basicExistence}
  The optimization problems \eqref{opa}, $\alpha>0$,  and \eqref{op}
  possess a unique minimizer
  $u^\alpha_\mathcal{\mathbb{X},\mathbb{Y}}\in \mathcal{H}_K$.
  For $\alpha>0$, the minimizer $u^\alpha_{\mathbb{X},\mathbb{Y}}$ is
  a weak solution of the equation 
\begin{equation}\label{eulera}
\mathcal{A}_K
u=\frac{1}{\alpha}[\mathbb{Y}-u(\mathbb{X})]^\top\delta_\mathbb{X}
=\frac{1}{\alpha}\sum_{k=1}^{|\mathbb{X}|}[y^k-u(x^k)]\delta_{x^k},
\end{equation}
i.e. a solution of the equation in $\mathcal{H}_K^*$. If
$\alpha=0$, then it holds that
$$
\mathcal{A}_Ku^0_{\mathbb{X},\mathbb{Y}}=\Lambda^\top
\delta_\mathbb{X}=\sum_{k=1}^{|\mathbb{X}|}\lambda_k \delta_{x^k}
$$
for some Lagrange multiplier $\Lambda\in
\mathbb{R}^{|\mathbb{X}|}$. The operator $\mathcal{A}_K$ is defined by
the validity of the identity $(u|\varphi)_{\mathcal{H}_K}=\langle
\mathcal{A}_Ku,\varphi\rangle_{\mathcal{H}_K^*,\mathcal{H}_K}$ for
$u,\varphi\in \mathcal{H}_K$, i.e. it coincides with the Riesz
isomorphism. The coefficient vector $\Lambda$ is obtained from the
augmented Lagrangian $E_0(u)-\Lambda^\top \bigl(
u(\mathbb{X})-\mathbb{Y}\bigr)$.
\end{prop}
\begin{proof}
  As the functionals are convex, existence and uniqueness readily
  follow (and is already known for $\alpha=0$ as discussed
  earlier). Taking variations in
  direction of functions $\varphi\in \mathcal{H}_K$ with
  $\varphi(\mathbb{X})=0$ shows that
  $\mathcal{A}_Ku^0_{\mathbb{X},\mathbb{Y}}\in N_\mathbb{X}^\perp$,
  where $N_\mathbb{X}=\big\{ \varphi\in \mathcal{H}_K\,\big |\,
  \varphi(\mathbb{X})=0\big\}$, which yields the representation
  formula that can also be obtained by taking variations with respect
  to the $u$ variable in the augmented Lagrangian to see that
  $$
  0=(u|\varphi)_{\mathcal{H}_K}-\Lambda^\top \varphi(\mathbb{X})=
  \langle \mathcal{A}_Ku-\Lambda^\top \delta_\mathbb{X},\varphi
  \rangle,\: \varphi\in \mathcal{H}_K.
  $$
  Taking variations of $E_\alpha$, $\alpha>0$, yields that
  $$
  0=(u|\varphi)_{\mathcal{H}_K}+\frac{1}{\alpha}\bigl(
  u(\mathbb{X})-\mathbb{Y}\bigr)\cdot\varphi(\mathbb{X})=
   \big\langle
   \mathcal{A}_Ku+\frac{1}{\alpha}(u(\mathbb{X})-\mathbb{Y})^\top
   \delta_\mathbb{X} ,\varphi\rangle,\: \varphi\in \mathcal{H}_K,
  $$
  which is the desired weak equation. 
\end{proof}

\begin{rem}\label{kernelAsFs}
  The identity
  $$
  \bigl( \mathcal{A}_K^{-1}\delta_x\bigr)(\tilde x)=\bigl(
  \mathcal{A}_K^{-1}\delta_x|k(\cdot,\tilde
  x)\bigr)_{\mathcal{H}_K}=\big \langle
  \mathcal{A}_K\mathcal{A}_K^{-1}\delta_x,k(\cdot,\tilde x)\rangle=k(x,\tilde x)
  $$
  identifies the kernel as a fundamental solution of the operator
  $\mathcal{A}_K$ leading to the Ansatz
  $$
  u=\Lambda^\top
  \mathcal{A}_K^{-1}\delta_\mathbb{X}=K(\cdot,\mathbb{X})\Lambda,
  $$
  and to the equation
  $\mathbb{Y}=u(\mathbb{X})=K(\mathbb{X},\mathbb{X})\Lambda$.
\end{rem}

\begin{rem}\label{fs}
The Laplace kernel $L(x)=e^{-|x|}$, $x\in \mathbb{R}^d$, is a fundamental solution of the
operator $c_d(1-4\pi^2\Delta)^{\frac{d+1}{2}}$ for
$c_d=\frac{\pi^{\frac{d+1}{2}}}{\Gamma(d+1)}$ and is therefore a
positive kernel (thanks to the characterization of positivity via
its Fourier transform, i.e. thanks to Bochner's Theorem, see
e.g. \cite{W04}) with native space given by 
$\operatorname{H}^{\frac{d+1}{2}}(\mathbb{R}^d)$. It is a kernel with
low regularity and is a special case of a kernel family known as
Mat\'ern kernels \cite{M86} that includes kernels of higher and higher regularity.
\end{rem}
\begin{proof}
This can be derived e.g. from the results about Bessel Potentials
found in \cite{AS61}.
\end{proof}
\begin{cor}
  The minimizer $u_{\mathbb{X},\mathbb{Y}}$ of \eqref{op} has the form
  $$
  u_{\mathbb{X},\mathbb{Y}}=K(\cdot,\mathbb{X})K(\mathbb{X},\mathbb{X})^{-1}\mathbb{Y}.
  $$
  The Ansatz
  $u^\alpha_{\mathbb{X},\mathbb{Y}}=K(\cdot,\mathbb{X})\Lambda$
  can be also used for the minimizer of \eqref{opa} and gives
  $$
  \Lambda=(\alpha+K(\mathbb{X},\mathbb{X})^{-1}\mathbb{Y}.
  $$
\end{cor}
\begin{proof}
  Plugging the Ansatz into the equation satisfied by
  $u^\alpha_{\mathbb{X},\mathbb{Y}}$ gives
  $$
  \Lambda^\top \delta_\mathbb{X}
  +\frac{1}{\alpha}[K(\cdot,\mathbb{X})\Lambda]^\top
  \delta_\mathbb{X}= \Lambda^\top \delta_\mathbb{X}
  +\frac{1}{\alpha}[K(\mathbb{X},\mathbb{X})\Lambda]^\top
  \delta_\mathbb{X}= \frac{1}{\alpha}\mathbb{Y}^\top\delta_\mathbb{X},
  $$
  which, by linear independence of evaluations at distinct points
  implies that
  $$
  \alpha \Lambda+K(\mathbb{X},\mathbb{X})\Lambda =\mathbb{Y},
  $$
  as claimed.
\end{proof}

Notice that the above results show how, in this particular setup, it
is possible to reduce the identification of a solution to the infinite
dimensional problem to the solution of a finite matrix problem.
\begin{rem}\label{ins}
It was observed in \cite{G25} that the optimization problems \eqref{op}
and \eqref{opa} in the special case when $K=L$ have infinite
dimensional counterparts that allow for a similar ``reduction in
size''. Indeed, letting $\mathcal{M}\subset\mathbb{R}^d$ be an
orientable compact smooth submanifold, then one can consider
$$
\operatorname{argmin}_{u\in \mathcal{H}_K}
\frac{1}{2}\left\{ \|
  u\|_{\mathcal{H}_K}^2+\frac{1}{\alpha}\int_\mathcal{M}
  (u-f)^2\, d \sigma_\mathcal{M}\right\}
$$
for $\alpha\geq 0$ with the understanding that, when $\alpha=0$ the
functional loses its second term that is replaced by the constraint
$u\big |_\mathcal{M}=f$. Arguments parallel to the ones used above
lead to the reduced equations
$$
\alpha u(x) +\int_\mathcal{M} L(x-y)u(y)\, d \sigma_\mathcal{M}(y)=
f(x),\: x\in \mathcal{M},
$$
on the manifold $\mathcal{M}$ (which replaces the cloud $\mathbb{X}$) for
$f:\mathcal{M}\to \mathbb{R}$ (in place of $\mathbb{Y}$). This analogy
shows, on the one hand, that kernel interpolation amounts to solving a
Fredholm integral equation of the first kind (i.e. a rather ill-posed
problem), while approximate interpolation corresponds to solving a
better conditioned second kind Fredholm integral equation. Regression,
in this light, appears to be a natural regularization of exact interpolation.
\end{rem}
There is a natural connection between \eqref{opa} and \eqref{op}.
\begin{prop}
Let $u^\alpha_{\mathbb{X},\mathbb{Y}}$ be the minimizer of \eqref{opa} for
$\alpha>0$ and of \eqref{op} for $\alpha=0$. Denote the corresponding
minima by $e_\alpha$, $\alpha\geq 0$, and let $0<\alpha_0<\alpha_1$. Then it holds that
\begin{enumerate}[(i)]
\item $u_{\mathbb{X},\mathbb{Y}}^\alpha\to
  u^0_{\mathbb{X},\mathbb{Y}}$ as $\alpha\searrow 0$ in $\mathcal{H}_K$.
\item $0\leq e_{\alpha_1}\leq e_{\alpha_0}\leq \frac{1}{2}\|
  u^0_{\mathbb{X},\mathbb{Y}}\|_{\mathcal{H}_K}^2$ and
  $|u^{\alpha_0}_{\mathbb{X},\mathbb{Y}}-\mathbb{Y}|^2\leq
  |u^{\alpha_1}_{\mathbb{X},\mathbb{Y}}-\mathbb{Y}|^2$.
\item $\| u^{\alpha_1}_{\mathbb{X},\mathbb{Y}}\|_{\mathcal{H}_K}\leq \|
    u^{\alpha_0}_{\mathbb{X},\mathbb{Y}}\|_{\mathcal{H}_K}\leq
    \|u^0_{\mathbb{X},\mathbb{Y}}\|_{\mathcal{H}_K}$ 
\item $|u^{\alpha}_{\mathbb{X},\mathbb{Y}}-\mathbb{Y}|^2\leq \alpha
  [\|u^0_{\mathbb{X},\mathbb{Y}}\|_{\mathcal{H}_K}^2-\|u^\alpha_{\mathbb{X},\mathbb{Y}}\|_{\mathcal{H}_K}^2]$.
\end{enumerate}
\end{prop}
\begin{proof}
  Notice that $e_\alpha\leq e_0$ for $\alpha>0$ since
  $u^0_{\mathbb{X},\mathbb{Y}}\in \mathcal{H}_K$
  and $E_\alpha(u^0_{\mathbb{X},\mathbb{Y}})=e_0$ where $E_\alpha$
  denotes the objective functional of \eqref{opa}. This entails that
  $\| u^\alpha_{\mathbb{X},\mathbb{Y}}\|_{\mathcal{H}_K}\leq C<\infty$ independently
  of $\alpha>0$.  In particular, we have that
  $$
  0\leq \frac{1}{2\alpha}|u^\alpha_{\mathbb{X},\mathbb{Y}}(\mathbb{X})-\mathbb{Y}|^2
  \leq e_\alpha-\frac{1}{2}\|u^\alpha_{\mathbb{X},\mathbb{Y}}\|_{\mathcal{H}_K}^2
  \leq e_0-\frac{1}{2}\|u^\alpha_{\mathbb{X},\mathbb{Y}}\|_{\mathcal{H}_K}^2=
  \frac{1}{2}\bigl[\|u^0_{\mathbb{X},\mathbb{Y}}\|_{\mathcal{H}_K}^2-
  \|u^\alpha_{\mathbb{X},\mathbb{Y}}\|_{\mathcal{H}_K}^2\bigr]
  $$
  Given any null sequence
  $(\alpha_k)_{k\in\mathbb{N}}$, weak compactness yields a weakly 
  convergent subsequence with limit $u^\infty_{\mathbb{X},\mathbb{Y}}\in
  \mathcal{H}_K$, which, by weak lower semicontinuity of the norm, must satisfy
  $\| u^\infty_{\mathbb{X},\mathbb{Y}}\|^2_{\mathcal{H}_K}\leq
  e_0$ and is therefore the unique minimizer
  $u^0_{\mathbb{X},\mathbb{Y}}$. Define now
  $d^2_{\mathbb{X},\mathbb{Y}}(u)=|u(\mathbb{X})-\mathbb{Y}|^2$
  and notice that
  \begin{align*}
    e_{\alpha_0}&=\frac{1}{2\alpha_0}d_{\mathbb{X},\mathbb{Y}}^2(
    u_{\mathbb{X},\mathbb{Y}}^{\alpha_0})+\|
                  u_{\mathbb{X},\mathbb{Y}}^{\alpha_0}\|^2_{\mathcal{H}_K} 
    \leq [\frac{1}{2\alpha_0}-\frac{1}{2\alpha_1}]
    d_{\mathbb{X},\mathbb{Y}}^2(u_{\mathbb{X},\mathbb{Y}}^{\alpha_1})+
    \| u_{\mathbb{X},\mathbb{Y}}^{\alpha_1}\|^2_{\mathcal{H}_K}+\frac{1}{2\alpha_1}
    d_{\mathbb{X},\mathbb{Y}}^2(u_{\mathbb{X},\mathbb{Y}}^{\alpha_1})\\
    &\leq
      [\frac{1}{2\alpha_0}-\frac{1}{2\alpha_1}]d_{\mathbb{X},\mathbb{Y}
     }^2(u_\mathcal{M}^{\alpha_1})+\| u_{\mathbb{X},\mathbb{Y}}^{\alpha_0}\|^2_{\mathcal{H}_K}
    +\frac{1}{2\alpha_1}d_{\mathbb{X},\mathbb{Y}}^2(u_{\mathbb{X},\mathbb{Y}
     }^{\alpha_0})
  \end{align*}
  yields that
  $d_{\mathbb{X},\mathbb{Y}}^2(u_{\mathbb{X},\mathbb{Y}}^{\alpha_0})\leq
  d_{\mathbb{X},\mathbb{Y}}^2(u_{\mathbb{X},\mathbb{Y}}^{\alpha_1})$. Since
  \begin{equation*}
    e_{\alpha_1}=\frac{1}{2\alpha_1}
    d_{\mathbb{X},\mathbb{Y}}^2(u_{\mathbb{X},\mathbb{Y}}^{\alpha_1})
    +\| u_{\mathbb{X},\mathbb{Y}}^{\alpha_1}\|^2_{\mathcal{H}_K}
    \leq\frac{1}{2\alpha_1}d^2_{\mathbb{X},\mathbb{Y}}(u)+
    \| u\|^2_{\mathcal{H}_K}\leq\frac{1}{2\alpha_0}d^2_{\mathbb{X},\mathbb{Y}}(u)
    +\| u\|^2_{\mathcal{H}_K},
  \end{equation*}
  taking the infimum over $u\in \mathcal{H}_K$
  gives $e_{\alpha_1}\leq e_{\alpha_0}$ as claimed. Next we have that
$$
\|u_{\mathbb{X},\mathbb{Y}}^{\alpha_1}\|^2_{\mathcal{H}_K}\leq
\|u_{\mathbb{X},\mathbb{Y}}^{\alpha_0}\|^2_{\mathcal{H}_K}-\frac{1}{2 \alpha_1}
\bigl[ d_{\mathbb{X},\mathbb{Y}}^2(u_{\mathbb{X},\mathbb{Y}}^{\alpha_1})
-d_{\mathbb{X},\mathbb{Y}}^2(u_{\mathbb{X},\mathbb{Y}}^{\alpha_0})\bigr] \leq
\|u_{\mathbb{X},\mathbb{Y}}^{\alpha_0}\|^2_{\mathcal{H}_K}
$$
thanks to the already established inequality for
$d_{\mathbb{X},\mathbb{Y}}^2(u_{\mathbb{X},\mathbb{Y}}^{\alpha})$. Finally,
it holds that (along the same subsequence of above for which we do not
introduce extra notation)
$$
\|u^0_{\mathbb{X},\mathbb{Y}}\|_{\mathcal{H}_K}=\|u^\infty_{\mathbb{X},\mathbb{Y}}\|_{\mathcal{H}_K}\leq
\liminf_{k\to\infty}\| u_{\mathbb{X},\mathbb{Y}}^{\alpha_k}\|_{\mathcal{H}_K}\leq
\limsup_{k\to\infty}\| u_{\mathbb{X},\mathbb{Y}}^{\alpha_k}\|_{\mathcal{H}_K}
\leq\|u^0_{\mathbb{X},\mathbb{Y}}\|_{\mathcal{H}_K}
$$
and all claims are established since convergence in a Hilbert space is
equivalent to weak convergence and convergence of the norm and we
proved that every sequence has a subsequence that converges to the
same limit.
\end{proof}
\begin{rem}
  While the above discussion revolved around the use of a positive
  definite kernel, it is useful and instructive to consider two
  extreme cases: the Laplace and the Gauss kernels. They represent a
  minimally smooth kernel (not even continuously differentiable), and
  a maximally smooth one (analytic). The latter one delivers higher order (and
  correspondingly worse conditioned) methods, while the first helps
  mitigating the intrinsic ill-posedness of exact interpolation. A
  slightly regularized version of the Laplace kernel will be used in
  numerical calculations for this very reason.
  We also notice that, in \cite{G25}, a derivation of Gauss kernel
  interpolation based on the use of the heat equation is proposed,
  which, at a purely formal level, amounts to replacing the Bessel potential norm in
  \eqref{op} and \eqref{opa} by the much  stronger norm
  $$
  \frac{1}{2}\| e^{-\Delta}u\|^2_{\operatorname{L}^2}= \frac{1}{2}\|
  e^{|\cdot|^2}\hat u\|^2_{\operatorname{L}^2}.
  $$
  While the use of arbitrary (positive definite) kernels is possible
  and useful, it sometimes obscures the exact nature of the
  differential operator that generates the norm on the reproducing
  kernel Hilbert space.
\end{rem}
It is Remark \ref{ins} that lead us to interpret problems
\eqref{op} and \eqref{opa} as representing interpolation and approximate
interpolation problems for functions $f$ defined on a manifold
$\mathcal{M}$, where $\mathcal{M}$ is only known through the
cloud/sample $\mathbb{X}$ and $f$ through the values $\mathbb{Y}$
which coincide with $f(\mathbb{X})$ or approximate it. In order to
access geometric information about $\mathcal{M}$ from $\mathbb{X}$, it
is then natural to use $u^\alpha_{\mathbb{X},\mathbb{1}}$, i.e. to choose
$f\equiv 1$, as a proxy for a defining function (or at least a local
level set function). In the rest of the paper, when performing
numerical experiments, we shall make use of the Gauss kernel as well
as the Laplace kernel, albeit in the regularized
form $L_\varepsilon(x)=e^{-\sqrt{|x|^2+\varepsilon}}$, $x\in
\mathbb{R}^n$, when needed. The former is protoypical of many a very
smooth kernel, while the second has minimal regularity and, even in
its regularized form $L_\varepsilon$, significantly alleviates the
ill-conditioning of the interpolation matrix as already mentioned
above. Kernels of any intermediate regularity between these extreme
cases are, e.g., those of the already mentioned Mat\'ern family.
\section{Hypersurfaces}
Given a point cloud or sample $\mathbb{X}$ of a compact smooth
hypersurface $\mathcal{M}$, we choose a kernel $K$, that is at least
twice continuously differentiable and compute the function
$$
u_\mathbb{X}=u_{\mathbb{X},\mathbb{1}}=
K(\cdot,\mathbb{X})K(\mathbb{X},\mathbb{X})^{-1}\mathbb{1},
$$
which amounts to a smooth extension of the constant function with
value 1 from $\mathbb{X}$ to the ambient space. As the kernel is
known, differential quantities associated to this function can be
computed analytically anywhere. Whenever $[u_\mathbb{X}=1]$
defines a hypersurface in the vicinity of one of its points, we can
define the implied normal vector to this surface by
$$
\nu_\mathbb{X} (x)= \frac{\nabla u_\mathbb{X}(x)}{|\nabla
  u_\mathbb{X}(x)|},
$$
where, again, the only numerical computation is that of the
coefficients $K(\mathbb{X},\mathbb{X})^{-1}\mathbb{1}$ since $\nabla
K(\cdot,\mathbb{X})$ can be computing exactly. Similarly one obtains
$$
D\nu_\mathbb{X}=\frac{1}{|\nabla u_\mathbb{X}|}\bigl( 
D^2u_\mathbb{X}-D^2u_\mathbb{X}\nu_\mathbb{X}\nu_\mathbb{X}^\mathsf{T}\bigr),
$$
i.e., a numerical approximation for the derivative 
$D\nu_\mathbb{X}$ of the normal in the spirit of meshfree methods. For
stability reasons and/or in the presence of noise, one can replace
$u_\mathbb{X}$ by
$u^\alpha_\mathbb{X}=u^\alpha_{\mathbb{X},\mathbb{1}}$ with a small
$\alpha>0$ or one related to the noise level, if known. We shall
revisit the choice of $\alpha>0$ for noisy data in the section where we
address the connection to Gaussian Process Regression. The quanties
$\nu_\mathbb{X}$ and $D\nu_\mathbb{X}$ will give us numerical access
to the tangent plane to the surface as well as to its principal
curvatures. This, in turn will open the door to the numerical
computation of geometric operators acting on functions defined on
$\mathbb{X}$. Notice that the deployment of
$u_{\mathbb{X},\mathbb{1}}$ replaces the use of the more commonly
encountered signed distance function (to the surface) which, however,
would require knowledge of the normal in order to be properly
interpolated. This is possible within the class of kernels chosen
here that are peaked at their center. The approach
also naturally offers a fix in the case of noisy data which consists
in making the estimation of geometric quantities robust by turning on
the regularization ($\alpha>0$). Estimates for the curvatures of a discrete
surface are of interest in many applications. We refer to \cite{Duclut22} for
the case of smooth (polygonal) data and to \cite{KaloEtAl07} for the noisy case
as well as to the references found in both for additional context.
\subsection{Interpolation Results}
It was observed in \cite{FW12} that the restriction of full space positive
definite kernels $K$ to a smooth compact embedded
submanifold $\mathcal{M}\subset \mathbb{R}^d$ without boundary of any
dimension is itself a 
positive kernel $K_\mathcal{M}:\mathcal{M}\times \mathcal{M}\to
\mathbb{R}$. They also identify the corresponding
native space of functions over $\mathcal{M}$ for kernels of known
finite smoothness and derive Sobolev 
error estimates for the interpolation problem for functions defined on
$\mathcal{M}$. Here we take a more direct approach to reach the same
conclusion. In essence, when $\mathbb{X}$ happens to be a sample of a
smooth manifold with the above properties, the full space interpolant
$u_{\mathbb{X},\mathbb{Y}}$ can be interpreted as a function in
$\mathcal{H}_K$, the trace
$\gamma_\mathcal{M}(u_{\mathbb{X},\mathbb{Y}})$ of which is an 
interpolant for $u:\mathcal{M}\to \mathbb{R}$ corresponding to the data
$(\mathbb{X},\mathbb{Y})$. Notice that, as long as $K$ has at least the
regularity of the Laplace kernel $L$, this apprach makes sense since
the corresponding native space is contained in
$\operatorname{H}^{\frac{d+1}{2}}(\mathbb{R}^d)\hookrightarrow
\operatorname{BUC}(\mathbb{R}^d)$. Thus full space
interpolation of data living on a manifold naturally leads to
approximations of functions belonging to the trace space of the native
space. This is in line with the results of \cite{FW12} and offers a slightly
different point of view. A difference between this paper and
\cite{FW12} is that we are also interested in using the data
$\mathbb{X}$ to interpolate the manifold itself in order to be able to
derive approximations for its geometric quantities and operators. This
is the main reason why we prefer to still think in terms of full space
kernel interpolation: defining functions of $\mathcal{M}$ do indeed
live in full space. We also observe that the minimally smooth case of the Laplace
kernel opens the door to effectively dealing with non-smooth manifolds as
the level sets of $\operatorname{H}^{\frac{d+1}{2}}(\mathbb{R}^d)$
functions do not enjoy much regularity. With the use of regularization
($\alpha>0$), the method can also be made robust even in the presence
of noise.
\begin{rem}
  We point out that, even in the minimally smooth case of the Laplace
  kernel, when working in the context of a low regularity submanifold
  $\mathcal{M}$, the level set function $u_{\mathbb{X},\mathbb{1}}$ is
  minimally smooth merely on $\mathcal{M}/\mathbb{X}$ as it holds
  that $\mathcal{A}_Ku_{\mathbb{X},\mathbb{1}}=0$ away from
  $\mathcal{M}/\mathbb{X}$, showing that $u_{\mathbb{X},\mathbb{Y}}\in
  \operatorname{dom}(\widetilde{\mathcal{A}}_K^n)$ for every $n\in
  \mathbb{N}$, where $\widetilde{\mathcal{A}}_K$ is the restriction of
  $\mathcal{A}_K$ to the open set
  $\mathcal{M}^\mathsf{c}/\mathbb{X}^\mathsf{c}$. Thus, in particular,
  the level sets other than the one with value 1 are typically
  smoother. This makes it possible to compute approximate values for
  geometric quantities even when the kernel lacks the necessary
  regularity.
\end{rem}
Let $u\in \operatorname{H}^m(\mathbb{R}^d)$ for $m>\frac{d}{2}$ be an
arbitrary function with $u=f$, where
$f\in\operatorname{H}^{m-1/2}(\mathcal{M})$ is given and $\mathcal{M}$ 
is assumed to be either a compact smooth hypersurface without boundary
in $\mathbb{R}^d$ or a finite set $\mathbb{X}\subset \mathcal{M}$. The
minimizer of an $\operatorname{H}^m(\mathbb{R}^d)$-norm functional
$E_m$ such as 
$$
E_m(u)=\frac{1}{2}\| u\|^2_{\operatorname{H}^{m}(\mathbb{R}^d)}
$$
with the constraint $u=f$ on $\mathcal{M}$ is a
relevant example in our context. When $\mathcal{M}$ is
a manifold, the minimizer is denoted by $u_{\mathcal{M},f}$ in line
with the notation used previously in the case when
$\mathcal{M}=\mathbb{X}$, $f=\mathbb{Y}$. Notice that, since $m>\frac{d}{2}$,
pointwise evaluations are possible and that the trace theorem ensures
that $\big\{ u\in \operatorname{H}^m(\mathbb{R}^d)\,\big |\, u=f\text{
  on }\mathcal{M} \big\}\neq\emptyset$.  We are
particularly interested in the case when $\mathbb{Y}=f(\mathbb{X})$,
which makes sense since the evaluation of $f$ is possible as
$m-\frac{1}{2}> \frac{d-1}{2}$ by assumption. The minimizers $u_{\mathcal{M},f}$ and
$u_{\mathbb{X},f(\mathbb{X})}=u_{\mathbb{X},\mathbb{Y}}$, with
$\mathbb{Y}=f(\mathbb{X})=\gamma_\mathcal{M}(u)(\mathbb{X})$, both
exist since the convex set determined by 
the constraint is non-empty and both belong to
$\operatorname{H}^m(\mathbb{R}^d)$. Thus it holds that 
$$
u_{\mathbb{X},\mathbb{Y}}\big |_\mathcal{M} \in
\operatorname{H}^{m-1/2}(\mathcal{M}) \text{ and }
u_{\mathbb{X},\mathbb{Y}}\big |_\mathcal{M}
(\mathbb{X})=f(\mathbb{X})
$$
It follows  that the function $u_{\mathbb{X},\mathbb{Y}}\big
|_\mathcal{M}-f$ vanishes on $\mathbb{X}\subset \mathcal{M}$. We can
therefore use results about the behavior of Sobolev functions with
scattered zeros as they can be found in \cite{BS94} for $\frac{d}{2}<m\in
\mathbb{N}$, or in \cite{W04} for $m\in (\frac{d}{2},\infty)\setminus
\mathbb{N}$.
While the results hold in the general context of $p\in[1,\infty)$, we
shall work with $p=2$ and $m>\frac{d}{p}=\frac{d}{2}$, when it holds
that
\begin{equation}\label{sobolevZeros}
|u|_{\operatorname{H}^k(\Omega)}\leq c\,
h_{\mathbb{X},\Omega}^{m-k}|u|_{\operatorname{H}^m(\Omega)}
\text{ and } \| u\|_\infty\leq c\, h_{\mathbb{X},\Omega}^{m-d/2}
|u|_{\operatorname{H}^m(\Omega)},
\end{equation}
for $u\in \operatorname{H}^m(\Omega)$ with $u(\mathbb{X})\equiv 0$ on
the assumption that $\Omega$ satisfy a geometric property which
always holds for balls and where $h_{\mathbb{X},\Omega}$ is the
so-called fill-distance of $\mathbb{X}$ given by
$$
h_{\mathbb{X},\Omega}=\sup_{x\in \Omega}\min_{\tilde x\in
  \mathbb{X}}|x-\tilde x|.
$$
It is a measure of how the set $\mathbb{X}$ approximates
$\Omega$. The norms appearing in the estimate are the semi-norms for the
Beppo Levi spaces $BL_k=\big\{ u\in \mathcal{D}'(\Omega)\,\big |\,
\partial^\alpha u\in \operatorname{L}_2(\Omega)\:\forall
|\alpha|=k\big\}$ given by
$$
|u|_{\operatorname{H}^k(\Omega)}^2=\sum_{|\alpha|=k}\|\partial^\alpha u\|_2^2.
$$
They are norms when $k>\frac{d}{2}$. The result is based on local
polynomial approximation. Transferring this result to the hypersurface
$\mathcal{M}$ yields the theorem below.
\begin{rem}
  Notice that that a more general result for Sobolev norms is
  formulated and proved in \cite{FW12} which allows $\mathcal{M}$ to
  be a submanifold of any dimension and the regularity to be
  fractional as well. Here we only provide a brief sketch of the proof
  specifically designed for our context for the sake of
  completeness. While formulated for hypersurfaces, it
  carries over verbatim to the case of submanifolds of any dimension.
\end{rem}
\begin{thm}
Let $\mathcal{M}\subset \mathbb{R}^d$ be a smooth compact
hypersurface and $f\in \operatorname{H}^{m-1/2}(\mathcal{M})$ be
given with $m>\frac{d}{2}$. Let $\mathbb{X}\subset \mathcal{M}$ and 
$$
h_{\mathbb{X},\mathcal{M}}=\sup_{x\in \mathcal{M}}\min_{\tilde x\in
  \mathbb{X}}d_\mathcal{M}(x,\tilde x)
$$
be the fill distance of the set $\mathbb{X}$ for the manifold
$\mathcal{M}$. If $f(\mathbb{X})\equiv0$, then it holds that
$$
|f|_{\operatorname{H}^k(\mathcal{M})}\leq c\,
h_{\mathbb{X},\mathcal{M}}^{m-1/2-k}|f|_{\operatorname{H}^{m-1/2}(\mathcal{M})}
\text{ and }\| f\|_\infty\leq c
h_{\mathbb{X},\mathcal{M}}^{m-d/2}|f|_{\operatorname{H}^{m-1/2}(\mathcal{M})}.
$$
\end{thm}
\begin{proof}
As $\mathcal{M}$ is assumed to be compact, it is possible to find a
finite number of open chart domains $U_j\subset \mathcal{M}$,
$j=1,\dots, J$, and (bijective) charts (local coordinates)
$$
\varphi_j:U_j\to \mathbb{B}_{\mathbb{R}^{d-1}}(0,1),
$$
such that $\varphi_j^{-1}: \mathbb{B}_{\mathbb{R}^{d-1}}(0,1)\to
\mathbb{R}^d$ is smooth and $\varphi_j^{-1}\bigl(
\mathbb{B}_{\mathbb{R}^{d-1}}(0,1)\bigr)=\mathcal{M}\cap
U_j=U_j$. Choosing a smooth partition of unity $\big\{
\psi_j:\mathcal{M}\to\mathbb{R}\,\big |\, j=1,\dots, J\big\}$
subordinate to this cover by chart domains, it is possible to obtain
an (equivalent) norm on $\operatorname{H}^{m-1/2}(\mathcal{M})$ given
by
$$
\| v\|_{\operatorname{H}^{m-1/2}(\mathcal{M})}\sim \bigl(
\sum_{j=1}^J\|
(v\psi_j)\circ\varphi_j^{-1}\|^2_{\operatorname{H}^{m-1/2}(\mathbb{B}_{\mathbb{R}^{d-1}
   (0,1)})}\bigr) ^{1/2},
$$
where $v\in \operatorname{H}^{m-1/2}(\mathcal{M})$. Next define
$\mathbb{X}_j=\mathbb{X}\cap U_j$ and
$\mathbb{Z}_j=\varphi_j(\mathbb{X}_j)$ for $j=1,\dots, J$. It holds
that
$$
h_{\mathbb{Z}_j, \mathbb{B}_{\mathbb{R}^{d-1}(0,1)}}=\sup_{z\in \mathbb{B}_{\mathbb{R}^{d-1}(0,1)}}
\min _{\tilde z\in \mathbb{Z}_j}|z-\tilde z|=\sup_{x\in U_j}\min
_{\tilde x\in \mathbb{X}_j}|\varphi_j(x)-\varphi_j(\tilde x)|\leq c_j
h_{\mathbb{X}_j,U_j}\leq c\, h_{\mathbb{X},\mathcal{M}},
$$
where $h_{\mathbb{Z}_j, \mathbb{B}_{\mathbb{R}^{d-1}(0,1)}}$ is the corresponding
fill distance for $\mathbb{Z}_j$ in 
$\mathbb{B}_{\mathbb{R}^{d-1}(0,1)}$. The last inequality can be
obtained as follows: we can assume without loss of generality that the
sets $U_j$ are geodesic balls centered at $p_j$ and with radii $r_j$
as well as that $h_{\mathbb{X},\mathcal{M}}<\min _{j=1,\dots,J}r_j$. Then, if $x\in
\mathbb{B}_\mathcal{M}(p_j,r_j-h_{\mathbb{X},\mathcal{M}})$, we can
find $\tilde x \in \mathbb{X}$ such that $d(x,\tilde x)\leq
h_{\mathbb{X},\mathcal{M}}$, in which case $\tilde x\in \mathbb{X}_j$.
Else, if $r_j-h_{\mathbb{X},\mathcal{M}}\leq
d_\mathcal{M}(x,p_j)<r_j$, we can find a point $\bar x$ along the
geodesic connecting $p_j$ to $x$ with $d(p_j,\bar
x)=r_j-h_{\mathbb{X},\mathcal{M}}$ and $\tilde x\in \mathbb{X}_j$ with
$d(\bar x,\tilde x)\leq h_{\mathbb{X},\mathcal{M}}$, in which case we
have
$$
d_\mathcal{M}(x,\tilde x)\leq d_\mathcal{M}(x,\bar
x)+d_\mathcal{M}(\bar x,\tilde x)\leq
r_j-(r_j-h_{\mathbb{X},\mathcal{M}})+h_{\mathbb{X},\mathcal{M}}\leq
2h_{\mathbb{X},\mathcal{M}}.
$$
Notice that
$|\mathbb{X}_j|=|\mathbb{Z}_j|$ since $\varphi_j$ is bijective and
that
\begin{align*}
h_{\mathbb{X}_j,U_j}&=\sup_{x\in U_j}\inf_{\tilde x\in \mathbb{X}_j}
d_\mathcal{M}(x,\tilde x)=\sup_{z\in
  \mathbb{B}_{\mathbb{R}^{d-1}(0,1)}}\inf_{\tilde z\in \mathbb{Z}_j}
 d_\mathcal{M}\bigl( \varphi^{-1}(z),\varphi^{-1}(\tilde z)\bigr) 
\\ &\leq c \sup_{z\in
  \mathbb{B}_{\mathbb{R}^{d-1}(0,1)}}\inf_{\tilde z\in
  \mathbb{Z}_j}|z-\tilde z|=c\,h_{\mathbb{Z}_j, \mathbb{R}^{d-1}(0,1)}.
\end{align*}
An application of \eqref{sobolevZeros} for each $j=1,\dots,J$, yields
$$
\| (f\psi_j)\circ\varphi_j^{-1}\|_{\operatorname{H}^k(\mathbb{B}_{\mathbb{R}^{d-1}
   (0,1)})}\leq c_j h_{\mathbb{Z}_j, \operatorname{H}^k(\mathbb{B}_{\mathbb{R}^{d-1}
   (0,1)})}^{m-1/2-k}\|
(f\psi_j)\circ\varphi_j^{-1}\|_{\operatorname{H}^{m-1/2}(\mathbb{B}_{\mathbb{R}^{d-1} 
   (0,1)})}
$$
and
$$
\| (f\psi_j)\circ\varphi_j^{-1}\|_\infty\leq c_j h_{\mathbb{Z}_j,
  \operatorname{H}^k(\mathbb{B}_{\mathbb{R}^{d-1} 
   (0,1)})}^{m-1/2-(d-1)/2}\|
(f\psi_j)\circ\varphi_j^{-1}\|_{\operatorname{H}^{m-1/2}
  (\mathbb{B}_{\mathbb{R}^{d-1}(0,1)})},
$$
and the claim follows by combining these inequalities and using the
relation between the fill distance on the manifold to that in
coordinate balls.
\end{proof}
\begin{cor}
Define $u_{\mathbb{X},f(\mathbb{X})}$ to be the
minimizer of the $\operatorname{H}^m(\mathbb{R}^d)$-norm subject to
the constraint that
$u_{\mathbb{X},f(\mathbb{X})}(\mathbb{X})=f(\mathbb{X})$. Then the function
$$
u_{\mathbb{X},f(\mathbb{X})}\big |_\mathcal{M}-f\in
\operatorname{H}^{m-1/2}(\mathcal{M})
$$
vanishes on $\mathbb{X}$ and it holds that
$$
\| u_{\mathbb{X},f(\mathbb{X})}\big |_\mathcal{M}-f\|
_{\operatorname{H}^k(\mathcal{M})}\leq c
h_{\mathbb{X},\mathcal{M}}^{m-1/2-k}\Big |u_{\mathbb{X},f(\mathbb{X})}\big
|_\mathcal{M}-f\Big | _{\operatorname{H}^{m-1/2}(\mathcal{M})}.
$$
as well as
$$
\| u_{\mathbb{X},f(\mathbb{X})}\big |_\mathcal{M}-f\|_\infty\leq c
h_{\mathbb{X},\mathcal{M}}^{m-d/2}\Big |u_{\mathbb{X},f(\mathbb{X})}\big
|_\mathcal{M}-f\Big |_{\operatorname{H}^{m-1/2}(\mathcal{M})}.
$$
Furthermore, since $u_{\mathbb{X},f(\mathbb{X})}$ and
$u_{\mathcal{M},f}$ are $\operatorname{H}^m(\mathbb{R}^d)$-extensions
of $u_{\mathbb{X},f(\mathbb{X})}\big |_\mathcal{M}$ and $f$ with
continuous dependence by construction, it also holds that
$$
\| u_{\mathbb{X},f(\mathbb{X})}-u_{\mathcal{M},f}\|_{\operatorname{H}^{k+1/2}(\mathbb{R}^d)}
\leq c\, h_{\mathbb{X},\mathcal{M}}^{m-1/2-k}\|
u_{\mathbb{X},f(\mathbb{X})}\big |_\mathcal{M}-f\|_{\operatorname{H}^{m-1/2}(\mathcal{M})}.
$$
\end{cor}
In kernel interpolation it is sometimes the case that one uses a
positive definite kernel $K:\Omega\times \Omega\to \mathbb{R}$,
$\Omega\subset \mathbb{R}^d$, without
an explicit characterization of the associated native space
$\mathcal{H}_K$. As we saw earlier,
$u_{\mathbb{X},\mathbb{Y}}=K(\cdot,x)K(\mathbb{X},\mathbb{X})^{-1}\mathbb{Y}$ 
is an interpolant for the values
$\mathbb{Y}\subset\mathbb{R}^{|\mathbb{X}|}$ at the arguments
$\mathbb{X}\subset \Omega$ and has minimal $\mathcal{H}_K$-norm among
all interpolants. Estimates for functions with scattered zeros like
the above in that context can then be derived as soon as one knows
that $\mathcal{H}_K \hookrightarrow \operatorname{H}^m(\Omega)$. Thus
the estimates obtained above can also be used for e.g Gauss kernel $G$
provided $f$ is a smooth enough function ensuring
solvability of the interpolation problem. This is due to the fact that
$\mathcal{H}_G \hookrightarrow \operatorname{H}^m(\mathbb{R}^d)$ for
every $m\in \mathbb{N}$. As the Gauss kernel is analytic, when the
manifold and the data are analytic, one would expect even better
convergence results along the lines of \cite[Section 11.4]{W04}.

If one is only interested in $\mathcal{M}$ and its geometry, then $f\equiv 1$
and the estimates are available on appropriate regularity assumptions
on $\mathcal{M}$ only. While the above theoretical results are clearly
valid for the Laplace kernel, its lack of smoothness does not allow for high
order approximations. Numerical experiments will show that high order
can be recovered by switching to the regularized kernel $L_\varepsilon$
mentioned before while, simultaneously, curbing the ill-posedness of the
system somewhat.

\begin{rem}
Kernel interpolation can deliver high order methods but, inevitably,
this comes at the cost of ill-posedness. The latter, in turn, used to
be thought of as a limitation to the size of point clouds that could
be considered. It has
since been observed that the use of compactly supported radial basis
functions \cite{Mo01,Ot03}, the use of the Multipole Method with
smooth kernel functions \cite{CBCMDMBE01,CBMFMM03}, or the use of the 
partition of unity method \cite{We02} can effectively mitigate this shortcoming.
\end{rem}

\begin{rem}
Kernel methods on special manifolds are also considered in the
literature (see for instance \cite{W04}) following an approach based
on kernels that live on the manifold itself and are related to its
geometry. As explicit kernels are only known for very special
manifolds, this leads to natural limitations.
\end{rem}
\subsection{Computation of Geometric Quantities}\label{geoqs}
Next we consider the problem of computing a numerical approximation to
the surface gradient $\nabla _\mathcal{M}f$ and the Laplace Beltrami
$\Delta _\mathcal{M}f$ operator of a function $f:\mathcal{M}\to
\mathbb{R}$ defined on a smooth compact hypersurface
$\mathcal{M}$. This is a problem of general interest due to its
manifold applications: surface fairing/smoothing/denoising, feature
detection, mesh optimization and simplification, geometric
compression, shape analysis and recognition, 3D reconstruction,
registration of surfaces, medical imaging, surface based morphology,
rendering, and PDEs on manifolds. It has been
widely studied in the literature in the context of various discrete
representations of surfaces (see for instance
\cite{Taubin95,MDSB03}). It was also considered for radial basis
functions representations \cite{Pi12}. As mentioned before, however, the
accepted approach is that of augmenting the data 
set $\mathbb{X}$ with two equidistant layers of additional points
$\mathbb{X}\pm \delta\nu_\mathcal{M}(\mathbb{X})$ located at a
distance $\delta>0$ from $\mathbb{X}$ in direction of the
positive and negative normals\footnote{Since \cite{Ho92}, the normals
  are typically approximated by Principal Component Analysis
  estimation of the tangent plane based on a ($k$-)nearest neighbors
  vicinity of each point.} and impose values at these points in an
effort to make the interpolant be an approximation of a ``signed distance
function'' for the manifold sampled by $\mathbb{X}$. The appearance of
disconnected components in the level sets of the interpolant along
with added stability are often adduced as the reasons justifying the
need of such an augmentation \cite{Pi12}. We believe (and will show
evidence in the last section) that disconnections are caused by the
use of very smooth (analytic) kernels, which, by design, try to
interpolate sparse, less smooth, or noisy hypersurfaces by highly
regular functions, which tend to use complex multi-component level
sets so as to accommodate the sample while maintaining regularity. We
will show that the use of low regularity (or close to such) kernels
and/or regularization in the sense advocated here (via regression) can
fix this problem. It should also be mentioned that some knowledge of
the normal to the surface is required in order to effectively
implement an augmentation strategy. Additionaly, exogenus parameters
are introduced that affect the approximation quality and the
conditioning of the resulting 
systems. These parameters also need to be determined 
somehow. While accuracy is lost (short of exact knowledge of the
normals), this approach reduces the 
computation of geometric operators to the computation of ambient space
operators applied to ambient space functions obtained by constant
extension in normal direction to the off-surface points (see
e.g. \cite{Pi12}). 

We assume again that $\mathcal{M}$ is known only through a sample
$\mathbb{X}\subset \mathcal{M}$ of its points and that the values of $f$ are only
available for $\mathbb{X}$, i.e. we assume $f(\mathbb{X})=\mathbb{Y}$
to be known. The starting point is the ability that we have gained of
obtaining geometric information about a hypersurface from a discrete
sample of its points via the associated signature function
$u_\mathbb{X}$. The algorithm is based on performing
computations on the interpolant $u_{\mathbb{X},\mathbb{Y}}$ that is
viewed as an approximation to the extension $u_{\mathcal{M},f}$ of $f$
to the whole space. We therefore need to derive a formula for the
surface operators in terms of the embedding $\mathcal{M}\subset
\mathbb{R}^d$ (or in terms of $\mathbb{X}$) that we can apply to the
extension and then consider
appropriate ``discrete'' counterparts. As we were not able to find a
such formula for the Laplace-Beltrami operator in 
the standard differential geometry textbooks, we give a brief
derivation here.
\begin{deff}
  Let $U,V:\mathcal{M}\to T \mathcal{M}$ be smooth sections of the
  tangent bundle (i.e. vector fields on $\mathcal{M}$). Then $D_V$
  denotes the derivative along the vector field $V$, i.e.
$$
D_Vf(x)=\left.\frac{d}{dt}\right|_{t=0} f \bigl(
\varphi_V(t,x)\bigr),\: D_VU(x)=\left.\frac{d}{dt}\right|_{t=0}
U\bigl( \varphi_V(t,x)\bigr),\: x\in \mathcal{M},
$$
where $\varphi_V(t,\cdot)$ is the flow generated by $V$, i.e. it
solves the ordinary differential equation
$$
\begin{cases}\dot \gamma =V(\gamma),&t>0,\\
  \gamma(0)=x&\end{cases}
$$
where $\gamma:[0,\infty)\to \mathcal{M}$ is a curve on
$\mathcal{M}$. This is well-defined since $V(x)\in T_x \mathcal{M}$
for every $x\in \mathcal{M}$. Even if $U$ is a tangential vector
field, it will in general not hold that $D_VU(x)\in T_x
\mathcal{M}$. This motivates the definition of the {\em covariant
derivative (connection)}
$$
\nabla_VU=\partial_V U=D_VU-\bigl(\nu\cdot D_VU\bigr)\nu,
$$
obtained by applying the orthogonal projection
$\mathbb{1}_d-\nu\nu^\top$  onto the tangent space $T\mathcal{M}$ of
$\mathcal{M}$ to the directional derivative $D_VU$.
\end{deff}
Now take $u_{\mathcal{M},f}:\mathbb{R}^d\to
\mathbb{R}$ or any other (smooth) extension of the function
$f:\mathcal{M}\to\mathbb{R}$ of interest. Then it holds that
$$
\nabla_\mathcal{M} u(y)=\nabla
u(y)-\bigl(\nu_\mathcal{M}(y)\cdot\nabla u(y)\bigr)\nu_\mathcal{M}(y),
$$
where $\nabla=\nabla_{\mathbb{R}^d}$ is the canonical gradient in the
ambient space. In order to compute the Laplace-Beltrami operator for
$\mathcal{M}$
$$
\Delta _\mathcal{M}=\operatorname{div}_\mathcal{M} \nabla_\mathcal{M}
$$
in terms of the coordinates of the ambient space $\mathbb{R}^d$,
we need the following lemma.
\begin{lem}
  For a smooth function $u:\mathbb{R}^d\to \mathbb{R}$ it holds that
$$
\Delta_\mathcal{M} u(x)=\operatorname{tr}(D^2u(x))-\nu(x)^\top
D^2u(x)\nu(x)+(d-1)H(x) \partial_{\nu(x)} u(x),\: x\in \mathcal{M}.
$$
\end{lem}
\begin{proof}
  If $X:\mathcal{U} \to \mathbb{R}^{d-1}$ is a local coordinate system
  for $\mathcal{M}$ about a point $y\in \mathcal{M}$, then the tangent
  space is spanned by the vectors
$$
\frac{\partial}{\partial x^j}= \bigl(\frac{\partial}{\partial
  x^j}X^{-1} \bigr)\bigl( X(y)\bigr),\: j=1,\dots, d-1,
$$
where it is typically assumed that $X(y)=0$ when working in a
neighborhood $\mathcal{U}$ of $y$. It follows that
$$
\nabla_{\frac{\partial}{\partial x^j}}\frac{\partial}{\partial
  x^k}=\frac{\partial}{\partial x^j}\frac{\partial}{\partial x^k}-
\bigl(\frac{\partial}{\partial x^j}\frac{\partial}{\partial x^k} \cdot
\nu\bigr) \nu,
$$
but $\frac{\partial}{\partial x^k}\cdot\nu=0$ and thus
$$
\frac{\partial}{\partial x^j}\frac{\partial}{\partial x^k}
\cdot\nu=-\frac{\partial}{\partial x^j}\nu\cdot
\frac{\partial}{\partial x^k}.
$$
It follows that
$$
\nabla_{\frac{\partial}{\partial x^j}}\frac{\partial}{\partial x^k}
=\frac{\partial}{\partial x^j}\frac{\partial}{\partial x^k}+
\frac{\partial}{\partial x^k}\cdot\frac{\partial}{\partial x^j}\nu
=\frac{\partial}{\partial x^j}\frac{\partial}{\partial x^k}+
\frac{\partial}{\partial x^j}\cdot\frac{\partial}{\partial x^k}\nu,
$$
where the last identity follows from
$\frac{\partial}{\partial x^j}\cdot\nu=0$ in view of the commutativity
of the regular derivatives and hence from
$$
\frac{\partial}{\partial x^k}\cdot\bigl( \frac{\partial}{\partial
  x^j}\nu\bigr)=-\frac{\partial}{\partial x^k}
\frac{\partial}{\partial x^j}\cdot\nu =\frac{\partial}{\partial x^j}
\cdot\bigl(\frac{\partial}{\partial x^k}\nu\bigr).
$$
Summarizing we obtain that
$$
\nabla_{\frac{\partial}{\partial x^j}}\nabla_{\frac{\partial}{\partial
    x^k}}u=\nabla_{\frac{\partial}{\partial
    x^j}}\frac{\partial}{\partial x^k}u=\frac{\partial}{\partial
  x^j}\frac{\partial}{\partial x^k}u+\bigl(\frac{\partial}{\partial
  x^j}\cdot \frac{\partial}{\partial x^k}\nu\bigr)\partial_\nu u.
$$
Now, take a system of local coordinates that yields an orthonormal
basis at $y$ for the tangent space. Observe that the addition of
$\nu(y)$ extends it to an orthonormal basis of $\mathbb{R}^d$, so
that, taking a trace in these coordinates at $y$, we arrive at
\begin{align}\notag
  \Delta_\mathcal{M} u=\nabla_{\frac{\partial}{\partial x^j}}\nabla_{\frac{\partial}{\partial
  x^j}}u&=\frac{\partial}{\partial x^j}\frac{\partial}{\partial
          x^j}u+\bigl(\frac{\partial}{\partial x^j}\cdot
          \frac{\partial}{\partial x^j}\nu\bigr)\partial_\nu u\\
        &=\operatorname{tr}(D^2u)-\nu^\top D^2u\nu+(d-1)H \partial_\nu u,
\end{align}
(where the summation convention was used) noticing that the mean
curvature $H$ satisfies
$$
(d-1)H=\operatorname{tr}(D\nu),
$$
since $\partial_\nu \nu=0$.
\end{proof}
In order to obtain a viable numerical algorithm, it remains only to
compute the interpolants $u_\mathbb{X}$ and
$u_{\mathbb{X},\mathbb{Y}}$ corresponding to the data set 
$(\mathbb{X},\mathbb{Y})$. The first is used in order to compute
the implied normal $\nu_\mathbb{X}$ and the implied mean curvature
$H_\mathbb{X}=\frac{1}{d-1}\operatorname{tr}(D\nu_\mathbb{X})$ as
described before, whereas the second yields a smooth extension of
$f\big |_\mathbb{X}$ to the full ambient space that is a numerical
approximation of the extension $u_{\mathcal{M},f}:\mathbb{R}^d\to
\mathbb{R}$ of $f$ as explained  in the previous subsection. Then, the
formul\ae{} for surface gradient and Laplace-Beltrami operators
immediately yield numerical discretizations given by
$$
\nabla_\mathbb{X} \mathbb{Y} =\nabla_\mathbb{X}
u_{\mathbb{X},\mathbb{Y}}=\nabla
u_{\mathbb{X},\mathbb{Y}}-\bigl(\nabla u_{\mathbb{X},\mathbb{Y}}\cdot
\nu_\mathbb{X}\bigr) \nu_\mathbb{X},
$$
and
$$
\Delta_\mathbb{X} \mathbb{Y} =\Delta _\mathbb{X}
u_{\mathbb{X},\mathbb{Y}}=\operatorname{tr}(D^2
u_{\mathbb{X},\mathbb{Y}})-\nu_\mathbb{X}^\top
D^2u_{\mathbb{X},\mathbb{Y}}\nu_\mathbb{X}+(d-1)H_\mathbb{X} \nabla
u_{\mathbb{X},\mathbb{Y}}\cdot \nu_\mathbb{X},
$$
where all derivatives involved are computed analytically. These
expressions can be evaluated anywhere as they are combinations of
functions defined everywhere. Matrix discretizations are obtained by
evaluating the expressions $\nabla_\mathbb{X}\mathbb{Y}$ and
$\Delta_\mathbb{X}\mathbb{Y}$ at $\mathbb{X}$. It is also possible to
obtain non square representations evaluating at other discrete sets
$\widetilde{\mathbb{X}}$ and/or to replace the ``grid'' $\mathbb{X}$ by
substituting $\mathbb{Y}$ with
$u_{\mathbb{X},\mathbb{Y}}(\widetilde{\mathbb{X}})$ to arrive at
matrices that approximate the operators and take inputs defined on
$\widetilde{\mathbb{X}}$ and produce outputs defined on that same
grid. 
\section{Connection to Regression}
Before moving on to numerical experiments, we address the connection
to Gaussian Process Regression which yields an interpretation of the
regularization parameter $\alpha>0$.

\begin{deff}
  If $E=E(\Omega)\subset \operatorname{C}(\Omega)$ is a Banach space
  of functions, a random vector $W\in E$ is said to have Gaussian law
  if $e^*W=\langle e^*,W \rangle_{E^*,E}$ is normally distributed for
  each $e^*\in E^*$. It is called centered iff $\mathbb{E}[e^*W]=0$
  for each $e^*\in E^*$.
\end{deff}
Multivariate Gaussian random variables on $\mathbb{R}^d$ are the
finite dimensional counterpart. The
associated covariance function is given by
$$
K(x,\tilde x)=\mathbb{E}[W(x)W(\tilde x)],\: x,\tilde x\in E,
$$
and uniquely determines the law of the centered Gaussian random
element. It is naturally non-negative definite as
$$
0\leq \mathbb{E}\bigl([\sum_{i=1}^m
a_iW(x_i)]^2\bigr)=\sum_{i,j=1}^ma_ia_j
\mathbb{E}\big[W(x_i)W(x_j)\big]=
\sum_{i,j=1}^na_ia_jK(x_i,x_j)=a^\top K(\mathbb{X},\mathbb{X})a,
$$
for any $\mathbb{X}\subset \Omega$. Thus one can start with a positive
definite kernel $K:\mathbb{R}^d\times \mathbb{R}^d\to \mathbb{R}$ and 
obtain an associated Gaussian random vector $W\in
\operatorname{C}(\Omega)$. When confronted with a data set
$(\mathbb{X},\mathbb{Y})$ with $\mathbb{X}\subset \Omega$ one can make
the assumption that
$$
y=f(x)+\varepsilon
$$
where $f\sim N(0,K)$, i.e. one can use a prior for $f$ in the form of
a random function with vanishing mean and covariance function $K$ and
model the presence of noise by the additive white noise term
$\varepsilon$ (independent of $f$) of variance $\sigma^2>0$. Given a
value set $\mathbb{Y}$ for a set of arguments $\mathbb{X}\in \Omega$,
one can then try to compute the conditional value $Y$ at a new
argument $x\in \Omega$ given the observed data
$$
y\big | \bigl(x,f(\mathbb{X})=\mathbb{Y}\bigr).
$$
It is known, using standard properties of multivariate Gaussian
distributions under conditioning and sum, that, on the above
assumptions, the random variable $Y$ is also Gaussian and that its
expected value can be computed as
$$
\mathbb{E}[Y]=K(x,\mathbb{X})\bigl[
\sigma^2+K(\mathbb{X},\mathbb{X})\bigr]^{-1}\mathbb{Y},
$$
which is the exact kernel interpolation formula when $\sigma=0$ and
coincides with the approximate interpolation formula when one chooses
$\alpha=\sigma^2>0$.
Thus one can think of the amount of regularization needed to be given
by the amount of uncertainty in the data, if it is known. This also
shows that the choice of prior in Gaussian Process Regression,
i.e. the choice of correlation function, determines the regularity
properties of the approximate interpolant (i.e. the expectation of the
posterior).
\begin{rem}
An introduction to the statistical point of view of Gaussian Process
Regression is a topic covered e.g. in the online book \cite{RW06} in the
context of Machine Learning. For the above problem it is also
possible to determine the variance of $Y$ which is given by
$$
\mathbb{E}\bigl[ \big |Y-\mathbb{E}[Y]\big |^2\bigr] =
K(x,x)+\sigma^2-K(x,\mathbb{X})[\sigma^2
+K(\mathbb{X},\mathbb{X})]^{-1}K(\mathbb{X},x).
$$
\end{rem}
We refer to \cite{RW06} for the material necessary to fill the gaps in
the arguments of this section.
\section{Numerical Experiments}
\begin{rem}
While the signature function is almost always a defining function, at
least locally, there are circumstances in which it is
not. Interestingly, this is the case when the curvature vanishes in
the case of curves, and when the mean curvature vanishes for
surfaces. These are all situations when the local symmetries of the
manifold make the normal vector ``vanish analytically'' (as explained
in the example below) and almost vanish
numerically. In the next subsection we show that this is only the case
in the absence of global information and does therefore not preclude
the use of the method in the presence of some flatness.
\end{rem}
\subsection{Curves}
In this section we illustrate how the method performs on closed curves. In all
figures, unless otherwise stated, we use white dots to indicate the data
set $\mathbb{X}$, a red line to indicate the level 1 of the signature
function $u_\mathbb{X}$, shades of green for its level regions, and
black arrows for the implied normals $\nu_\mathbb{X}$. We begin by
finishing the above discussion. We take a closed smooth curve
$\Gamma\subset \mathbb{R}^2$ parametrized by
$$
\bigl( \sin(t),\sin^3(t)+.5\cos(t)\bigr) ,\: t\in[0,2\pi),
$$
which has two inflection points.
\begin{figure}[ht]
    \centering
    \includegraphics[width=0.45\textwidth]{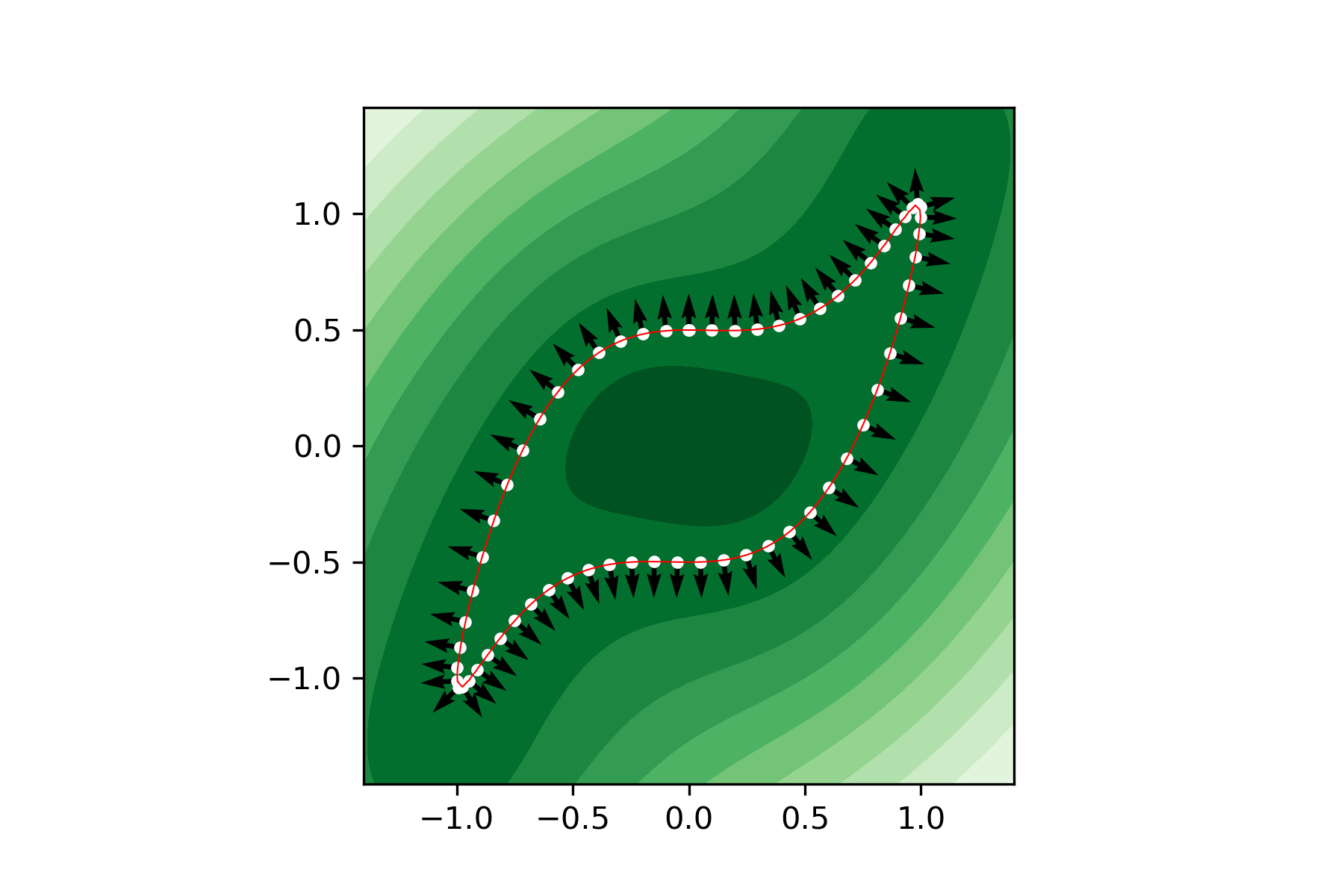}
    \includegraphics[width=0.45\textwidth]{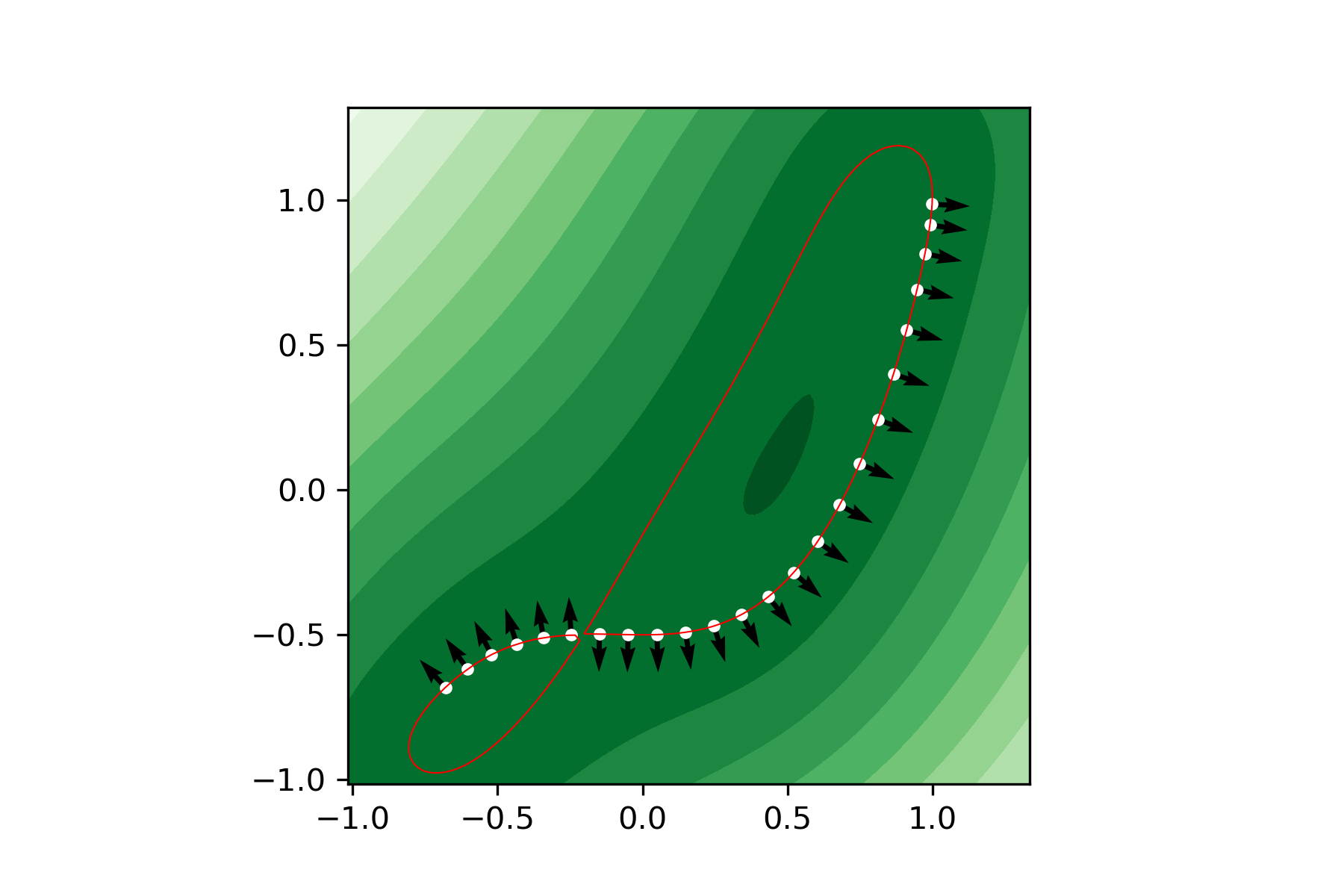}
    \caption{Reconstruction of a closed curve from 64 of its points
      (in white) on the left. Reconstruction based on an incomplete
      subsample consisting of 24 points on the right.}
    \label{F:cubicClosed}
\end{figure}
Figure \ref{F:cubicClosed} (left) shows the information recovered from
$u_\mathbb{X}$ where $\mathbb{X}$ consists 
of 64 points at equidistant parameter values. As you can see the
normals all point in an outward direction orthogonal to the
curve. However, if we remove a part of the closed curve, it obviously
becomes impossible to define regions interior and exterior to the
curve and this is reflected in the vanishing of the gradient of
$u_\mathbb{X}$ at the one inflection point found on this portion of
the curve as can be inferred from Figure \ref{F:cubicClosed}
(right). As it is to be expected from the contruction of
$u_\mathbb{X}$, the induced normal, when computed
based on local information alone, points in the direction of the
region from which the curve looks (locally) convex (and thus
effectively flips its sense as it goes through the inflection
point as if using a discontinuous parametrization). In this example we 
used interpolation with the Laplace kernel and $\varepsilon=1$. Notice
as well that the use of a smooth kernel still yields a smooth closed
curve that is necessarily the level set of an analytic function.

Next we consider a non-smooth curve, a triangle. This examples shows
that, while the Gaussian kernel can be used also for samples of
(simple) non-smooth curves, it will (as in the previous example)
generate complicated level lines so as to decompose the curve into
smooth pieces. The use of the Laplace kernel avoids this issue 
since the signature function on the set or its sample is only H\"older
and can therefore easily accommodate abrupt turns in the curve. Even
in more complicated non-smooth examples, Gauss interpolation yields
complicated and mostly useless level sets but, at the same time, still
delivers useful normals on the sample set as in the example of the
triangle. Figure \ref{F:triangle} shows the results for the Gauss
(left) and Laplace (right) kernels, where 48 sample points are used
and $\varepsilon=10^{-5}$.
\begin{figure}[ht]
    \centering
    \includegraphics[width=0.45\textwidth]{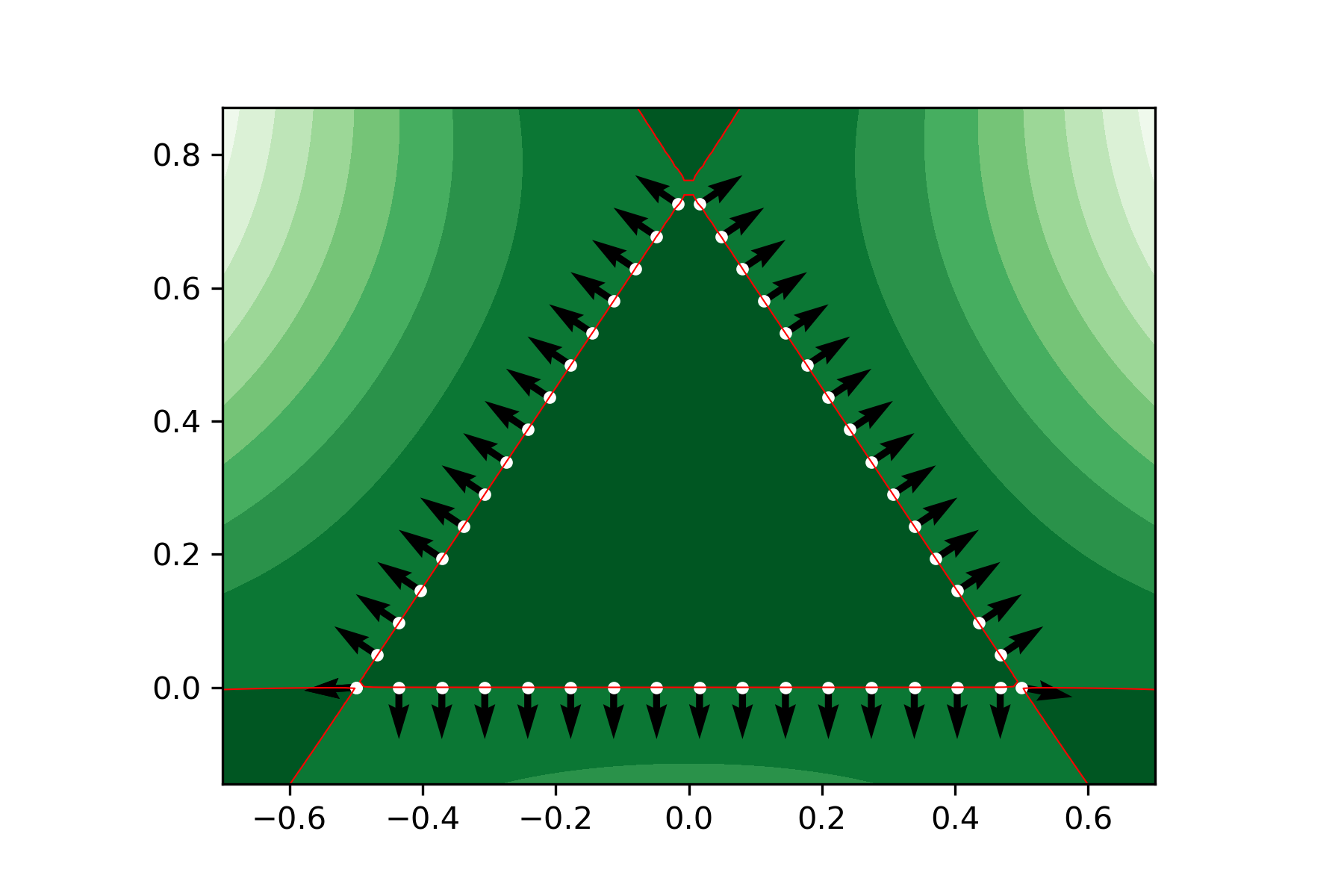}
    \includegraphics[width=0.45\textwidth]{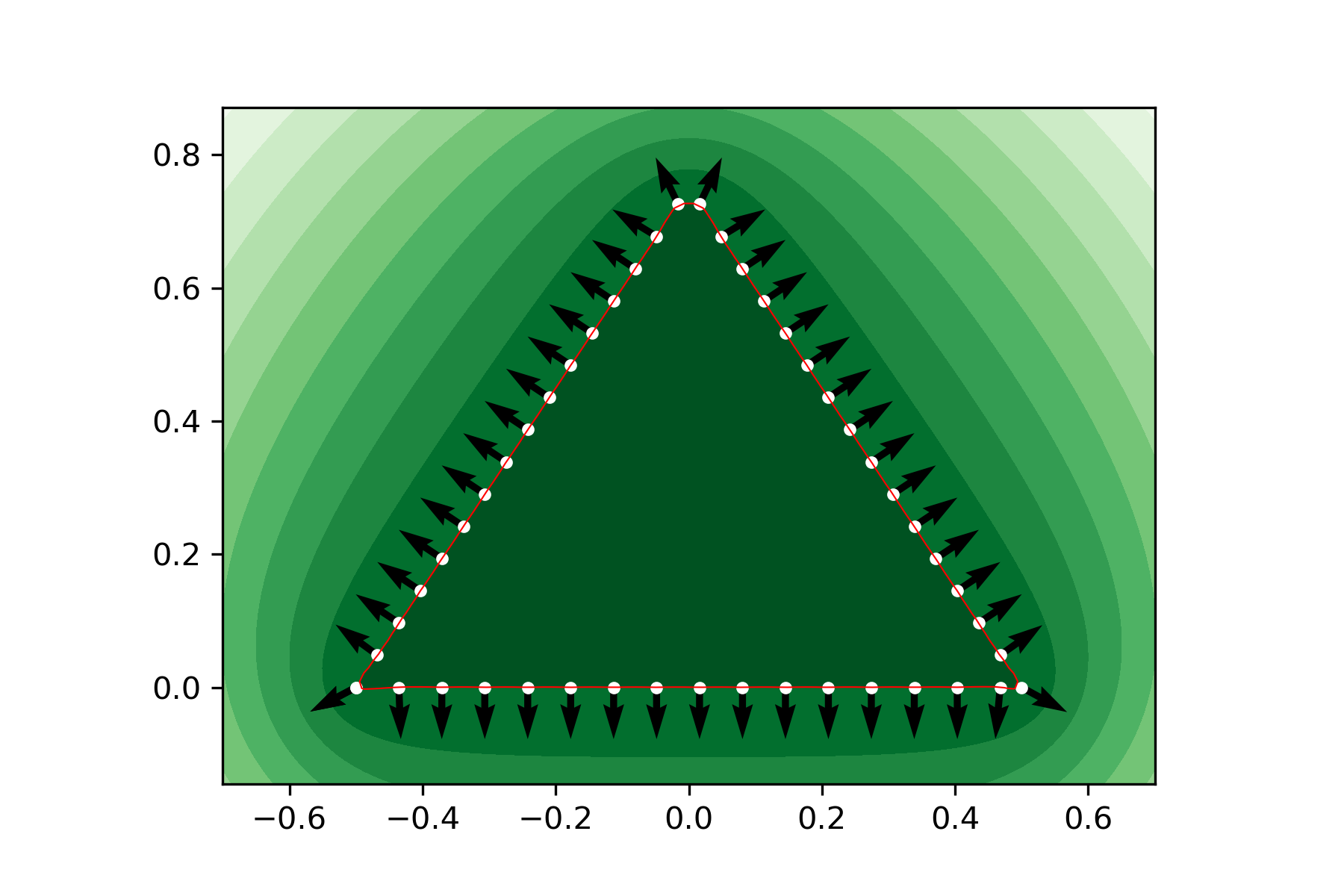}
    \caption{Reconstruction of a triangle from 48 sample points using
      Gauss and Laplace kernel interpolation, left and right,
      respectively.}
    \label{F:triangle}
\end{figure}
We illustrate the high accuracy of analytic kernel interpolation and
the associated artifacts in two examples of curves. In the first, see
Figure \ref{F:semiEllipse}, we
compute the implied level set of the signature function of a sample
consisting of points situated on one half of an ellipse, in the second (Figure
\ref{F:square}) that of a
full square and the same square with a missing corner. The first shows
the ability of the method to perform analytic continuation. The second
shows how the level sets of an analytic function attempt to accomodate
the sample of a non-smooth curve.
\begin{figure}[ht]
    \centering
    \includegraphics[width=0.45\textwidth]{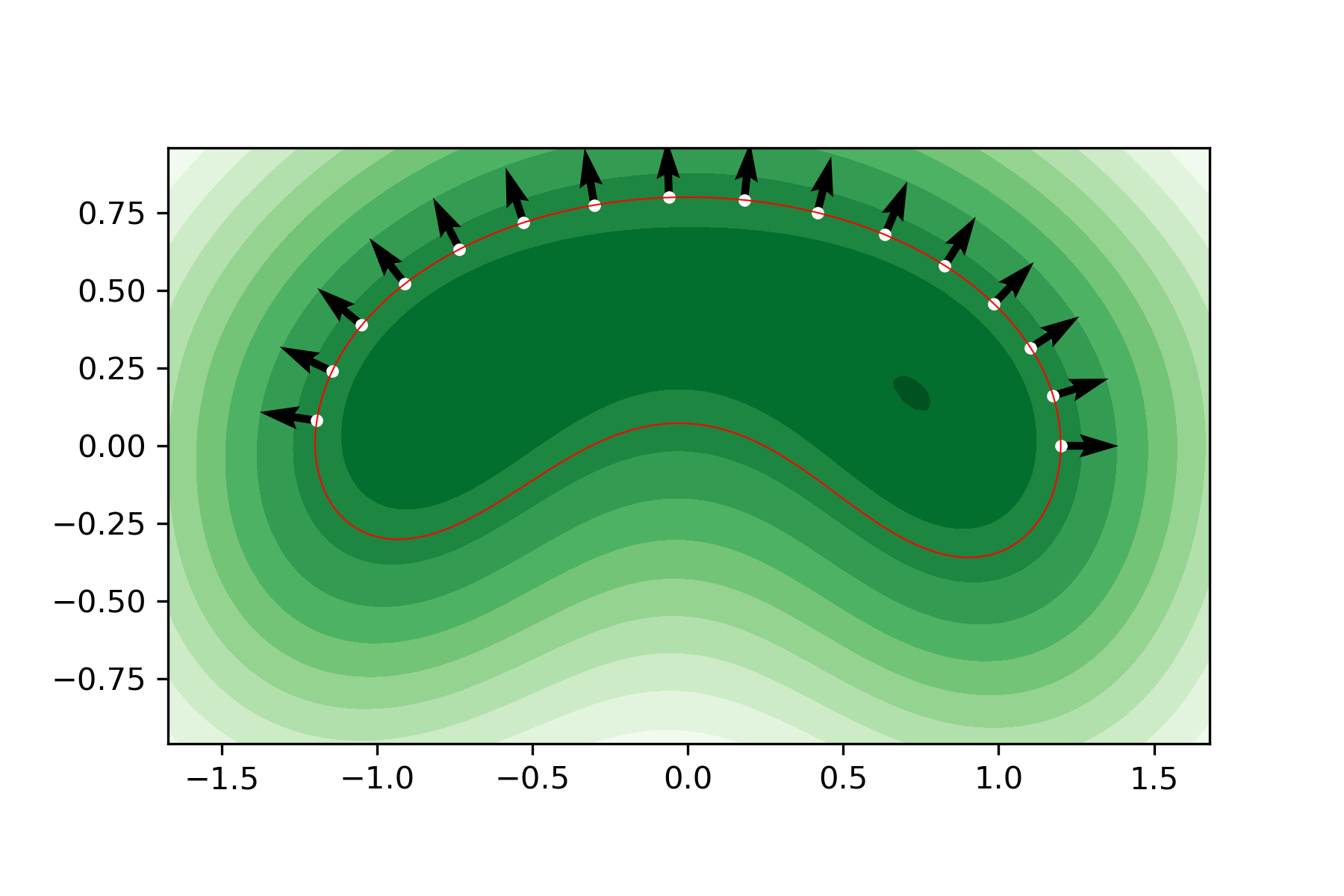}
    \includegraphics[width=0.45\textwidth]{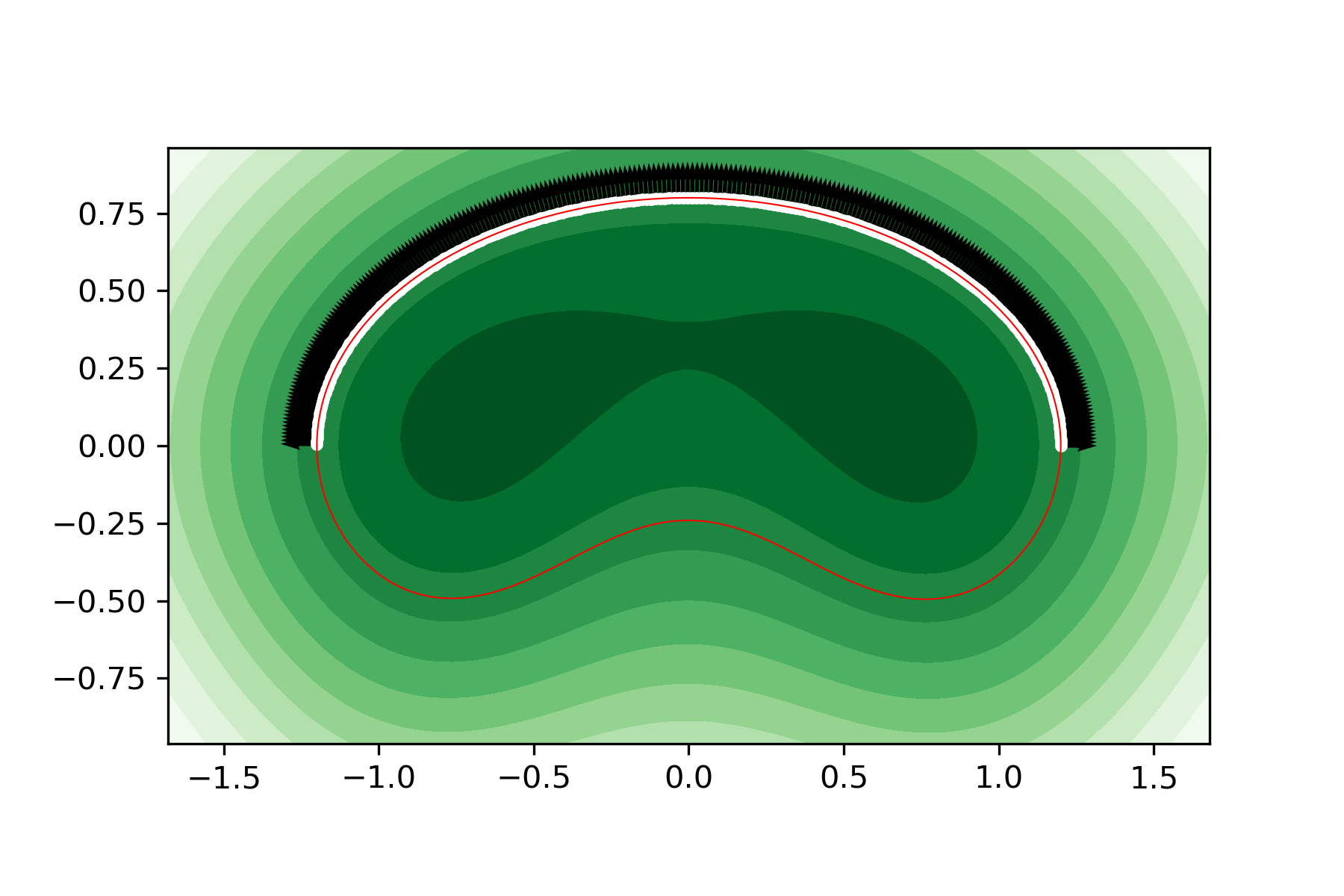}
    \caption{Implied curve of an exact sample consisiting of $m$
      points along half of a ellipse. Left: $m=16$. Right: $m=256$.}
    \label{F:semiEllipse}
  \end{figure}
  \begin{figure}[ht]
    \centering
    \includegraphics[width=0.45\textwidth]{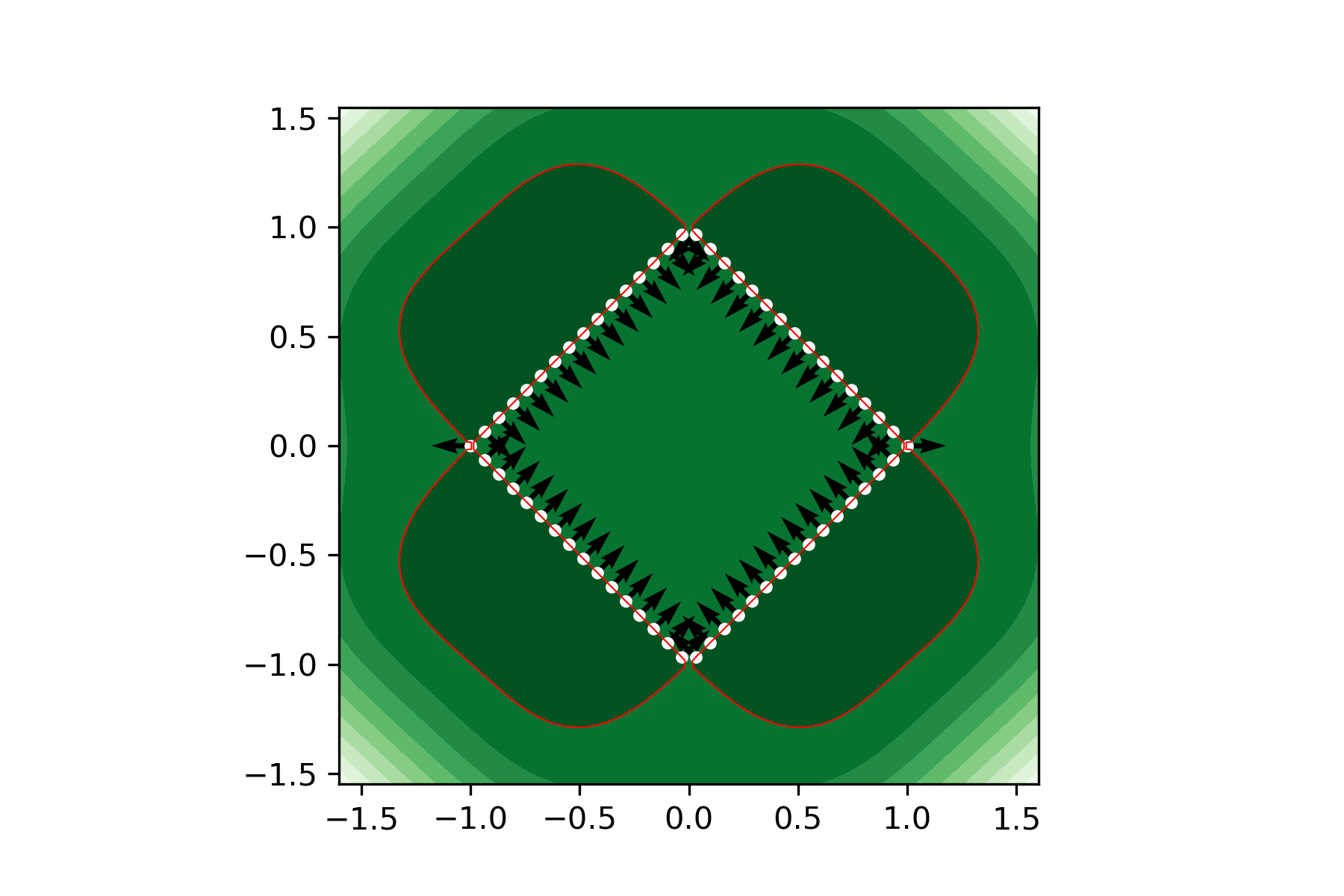}
    \includegraphics[width=0.45\textwidth]{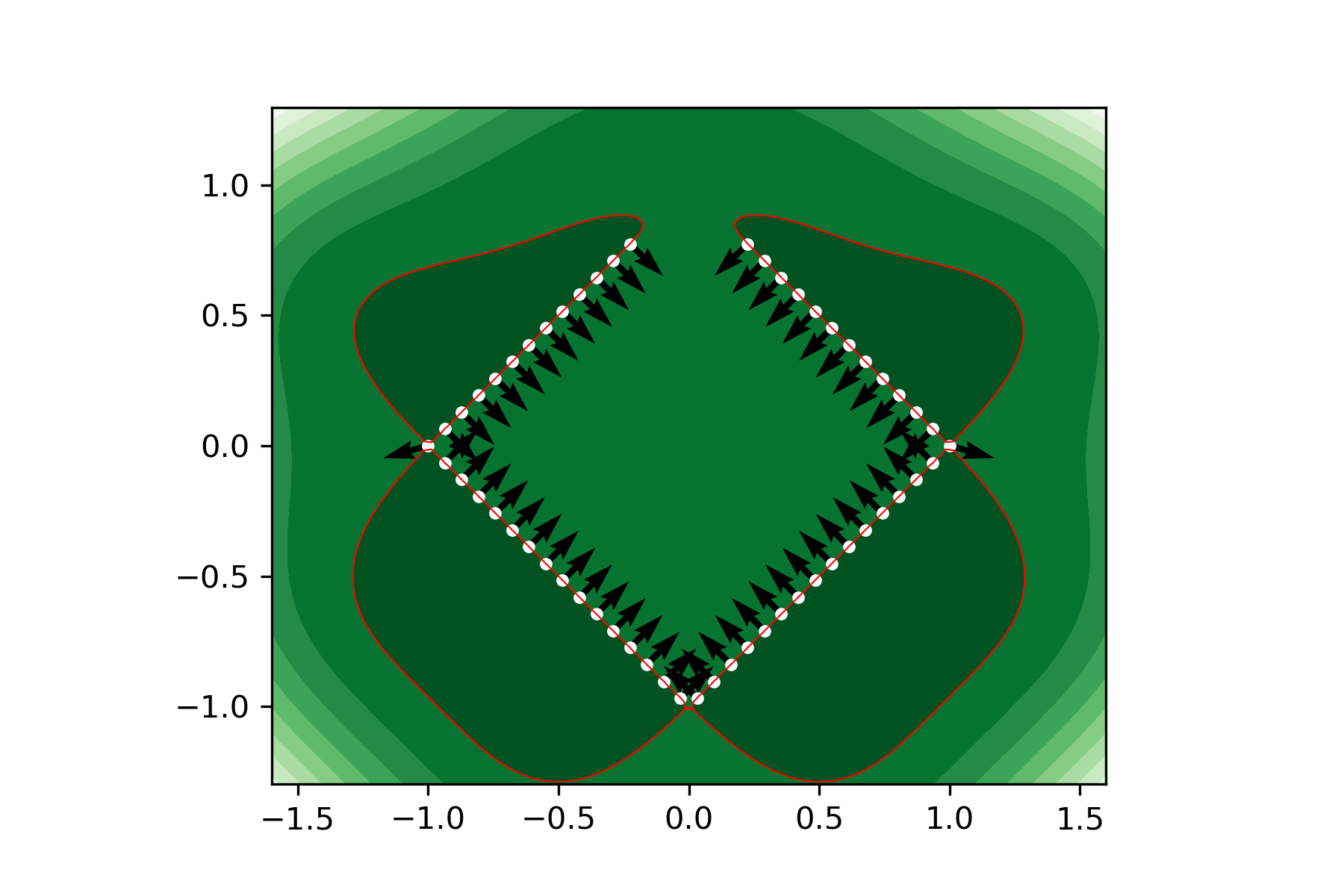}
    \caption{Implied curve of an exact sample of points of a full
      square and one with a corner removed.}
    \label{F:square}
  \end{figure}
These latter examples underscore how the disconnection problems reported in the
literature and observed when using radial basis functions to interpolate
scattered data is in good part due to the regularity of the chosen
kernel (and the implied regularity of the interpolant) and not, 
as sometimes argued \cite{CBCMDMBE01}, to the local nature  (read fast
decay) of the
kernel. In fact the problem can be all but resolved
by using less regular kernels or by regularizing in the way described
before.
\subsection{Quadratic Surfaces}
We briefly consider the case of quadratic surfaces (locally) in order
to revisit the remark about ``flat'' points and to show that the
method is capable of recovering an approximation to the actual
principal curvatures if the symmetry is broken. We consider the
surface defined by $z=\frac{a}{2}x^2 -\frac{b}{2}y^2$ ($a,b>0$) and
concentrate on the origin. When $a\neq b$, the method recovers an
approximation to the principal curvatures. Indeed, taking the patch
of the surface above $[-.5,.5]\times[-.5,.5]$ and using a regular
discretization of $16\times16$ points, Gauss and regularized
($\varepsilon=1$) Laplace yield the results of Table \ref{T:kappasA}.
\begin{table}[h]
\centering
\begin{tabular}{l|c|c}
Kernel & \( \kappa_1 \) & \( \kappa_2 \) \\\hline
Gauss   & \( 1.000 \times 10^{0} \) & \( -2.000 \times 10^{0} \) \\
Laplace & \( 1.003 \times 10^{0} \) & \( -1.992 \times 10^{0} \) \\
\end{tabular}
\caption{Quadratic surface. Principal curvatures in the origin computed using Gauss and
  Laplace kernels for $a=1$ and $b=2$ using 256 points.}
\label{T:kappasA}
\end{table}
The same computation with $a=b=1$ gives the results of Table
\ref{T:kappasB}. In this case, the symmetry leads to very small
gradients, their lengths are $3.1948\times 10^{-11}$ (Gauss) and
$5.6192\times 10^{-12}$ (Laplace) and, hence, to a very inaccurate
numerical estimation of the curvature.
\begin{table}[h]
\centering
\begin{tabular}{l|c|c}
Kernel & \( \kappa_1 \) & \( \kappa_2 \) \\\hline
Gauss   & $-4.4615\times 10^{10}$ & $-3.1650e\times 10^{5}$\\
Laplace &$-3.8913\times 10^9$  & $3.5405\times 10^{7}$\\
\end{tabular}
\caption{Quadratic surface. Principal curvatures in the origin computed using Gauss and
  Laplace kernels for $a=b=1$ using 256 points.}
\label{T:kappasB}
\end{table}
Interestingly, if one replaces the ``regular'' sample with a random
sample of the domain (uniform distribution), Laplace interpolation
(less ill-posed) is able to recover the actual curvatures. Indeed for
a random sample of the same size, one obtains
$$
\kappa_1=9.8534\times 10^{-1},\: \kappa_2=-1.0138\times 10^0,
$$
an approximation, which is consistent over many runs. As previously
mentioned this is a problem encountered when only local information is
available and disappears when dealing with closed smooth surfaces with
a well-defined inside and outside.
\subsection{Laplace-Beltrami Operator of the Sphere}
We consider now the numerical computation of the Laplace-Beltrami
operator of a function $f$ given through its values $\mathbb{Y}$ at
the $m$ points $\mathbb{X}$ of the Fibonacci discretization of the
unit sphere\footnote{The basic idea is to use the golden ratio to
  generate a non-closing spiral of points on the sphere to otain an
  almost uniform distribution (see \cite{SP06,Go10}).}. In Table
\ref{T:sphereLB}, we report
the average relative $L^\infty$-error at 32 
randomly chosen points on the sphere (and hence unrelated to the
discretization points) in the values of the function, its surface
gradient, and Laplace-Beltrami operator applied to it.
\begin{table}[h]
\centering
\begin{tabular}{|r|c|c|c|}\hline
\( m \) & $f$ error  & $\nabla_{\mathbb{S}^2}f$ error &$\Delta_{\mathbb{S}^2}f$ error \\
\hline
64   & \( 3.811 \times 10^{-2} \) & \( 3.857 \times 10^{-1} \) & \( 5.673 \times 10^{-1} \) \\\hline
128  & \( 8.538 \times 10^{-4} \) & \( 1.579 \times 10^{-1} \) & \( 1.654 \times 10^{-2} \) \\\hline
256  & \( 2.623 \times 10^{-6} \) & \( 4.271 \times 10^{-5} \) & \( 3.136 \times 10^{-5} \) \\\hline
512  & \( 1.117 \times 10^{-9} \) & \( 2.971 \times 10^{-8} \) & \( 1.268 \times 10^{-8} \) \\\hline
\end{tabular}
\caption{Average relative $L^\infty$-error at different discretization
levels $m$ for function evaluation (interpolation), surface nabla, and
Laplace-Beltrami evaluation.}
\label{T:sphereLB}
\end{table}
The function chosen is $f(x)=\sin \bigl( \pi(1+2x_3)\bigr)$ and we
work with the Laplace kernel with $\varepsilon=1$. The exact surface
nabla and Laplace-Beltrami can be computed
\begin{align*}
\nabla_{\mathbb{S}^2}f(x)&=2\pi\cos \bigl( \pi(1+2x_3)\bigr)
\begin{bmatrix}
 -x_1x_2&-x_2x_3&1-x_3
\end{bmatrix}^\top\\
\triangle_{\mathbb{S}^2}f(x)&=8\pi (x_3^2-1)\sin \bigl(
\pi(1+2x_3)\bigr)-4\pi x_3\cos \bigl( \pi(1+2x_3)\bigr)
\end{align*}
for $x\in \mathbb{S}^2$ and compared to the numerical values obtained
with the method developed in Section \ref{geoqs} using $u_\mathbb{X}$,
$\nu_\mathbb{X}$, $H_\mathbb{X}$, and $u_{\mathbb{X},\mathbb{Y}}$. We
underline the fact that the method has only knowledge of $f$ at a
sample of points on the sphere and that $u_{\mathbb{X},\mathbb{Y}}$ is
an extension of $f\big |_{\mathbb{X}}$ that, yes, interpolates
$f$ on $\mathbb{S}^2$ but can be quite different away from the
sphere. In particular its gradient and Hessian can have ``little'' to do with
those of $f$.

\subsection{The Ring Torus}
In this section we illustrate how the method computes the principal
curvatures of a torus parametrized by
$$
\left( [R_1+R_2\cos(v)]\cos(u),
[R_1+R_2\cos(v)]\sin(u),R_2\sin(u)\right),\: u,v\in[0,2\pi),
$$
where $R_1=2$ and $R_2=.5$. The curvatures are given by $\kappa_1=-2$ and
$\kappa_2=\frac{\cos(v)}{R_1+R_2\cos(v)}$, i.e. one is the curvature
of any of the vertical circles, while the other depends on the height
of the point considered as well as its distance to the origin. For
this reason we choose $u\in[0,2\pi)$ at random, fix it, and compute
the curvatures numerically at 32 equidistant points on the
corresponding vertical circle. The signature function $u_\mathbb{X}$
is computed for samples of size $|
\mathbb{X}|=64,128,256,512,1024$. These samples are obtained in two
different ways. The first consists in picking $u,v$ independently and
uniformly at random from $[0,2\pi)$ and then accepting the sample
$(u,v)$ with a probability corresponding to the normalized surface
area at the corresponding point (see \cite{foley96}). The second is based on the Fibonacci lattice,
which, while deterministic, does avoid clustering and is almost
uniform. For the first sampling method, Table \ref{T:torus} shows the
$\operatorname{L}^\infty$ and 
$\operatorname{L}^2$ relative errors in the two curvatures (ordered by
size) , while the samples are depicted in Figure \ref{F:torus}. In all
experiments we use $\alpha=10^{-10}$ for both kernels (thus
introducing a limit to the achievable accuracy) and
$\varepsilon=1$ for the Laplace kernel.
{\small
\begin{table}[ht]
\centering
\begin{tabular}{|c|c|c|c|c|}
\hline
\textbf{$m$} & \textbf{$L^\infty$ (Gauss)} & \textbf{$L^2$ (Gauss)} & \textbf{$L^\infty$ (Laplace)} & \textbf{$L^2$ (Laplace)} \\
\hline
64   & $1.85{\times}10^{-1}$ / $1.72{\times}10^{-1}$ & $9.88{\times}10^{-2}$ / $4.41{\times}10^{-2}$ & $1.02{\times}10^{-1}$ / $2.21{\times}10^{-1}$ & $2.20{\times}10^{-1}$ / $1.33{\times}10^{-1}$ \\
128  & $1.77{\times}10^{-2}$ / $1.61{\times}10^{-2}$ & $1.41{\times}10^{-2}$ / $5.24{\times}10^{-3}$ & $5.26{\times}10^{-2}$ / $1.08{\times}10^{-1}$ & $1.72{\times}10^{-2}$ / $9.56{\times}10^{-3}$ \\
256  & $5.09{\times}10^{-5}$ / $1.74{\times}10^{-4}$ & $8.33{\times}10^{-4}$ / $2.75{\times}10^{-4}$ & $8.61{\times}10^{-3}$ / $7.94{\times}10^{-3}$ & $2.01{\times}10^{-3}$ / $9.18{\times}10^{-4}$ \\
512  & $2.23{\times}10^{-8}$ / $1.19{\times}10^{-7}$ & $2.25{\times}10^{-7}$ / $7.86{\times}10^{-8}$ & $1.59{\times}10^{-3}$ / $9.46{\times}10^{-4}$ & $1.08{\times}10^{-3}$ / $3.64{\times}10^{-4}$ \\
1024 & $1.75{\times}10^{-8}$ / $1.36{\times}10^{-8}$ & $5.99{\times}10^{-8}$ / $2.33{\times}10^{-8}$ & $8.51{\times}10^{-6}$ / $4.18{\times}10^{-6}$ & $3.78{\times}10^{-6}$ / $1.39{\times}10^{-6}$ \\
\hline
\end{tabular}
\caption{Randomly sampled torus. Relative $L^\infty$ and $L^2$ errors for Gauss and Laplace
  kernels at different values of $m$. Each entry shows the error
  values for the two principal curvatures ordered according to size.}
\label{T:torus}
\end{table}}

\begin{figure}[ht]
    \centering
    \includegraphics[width=0.4\textwidth]{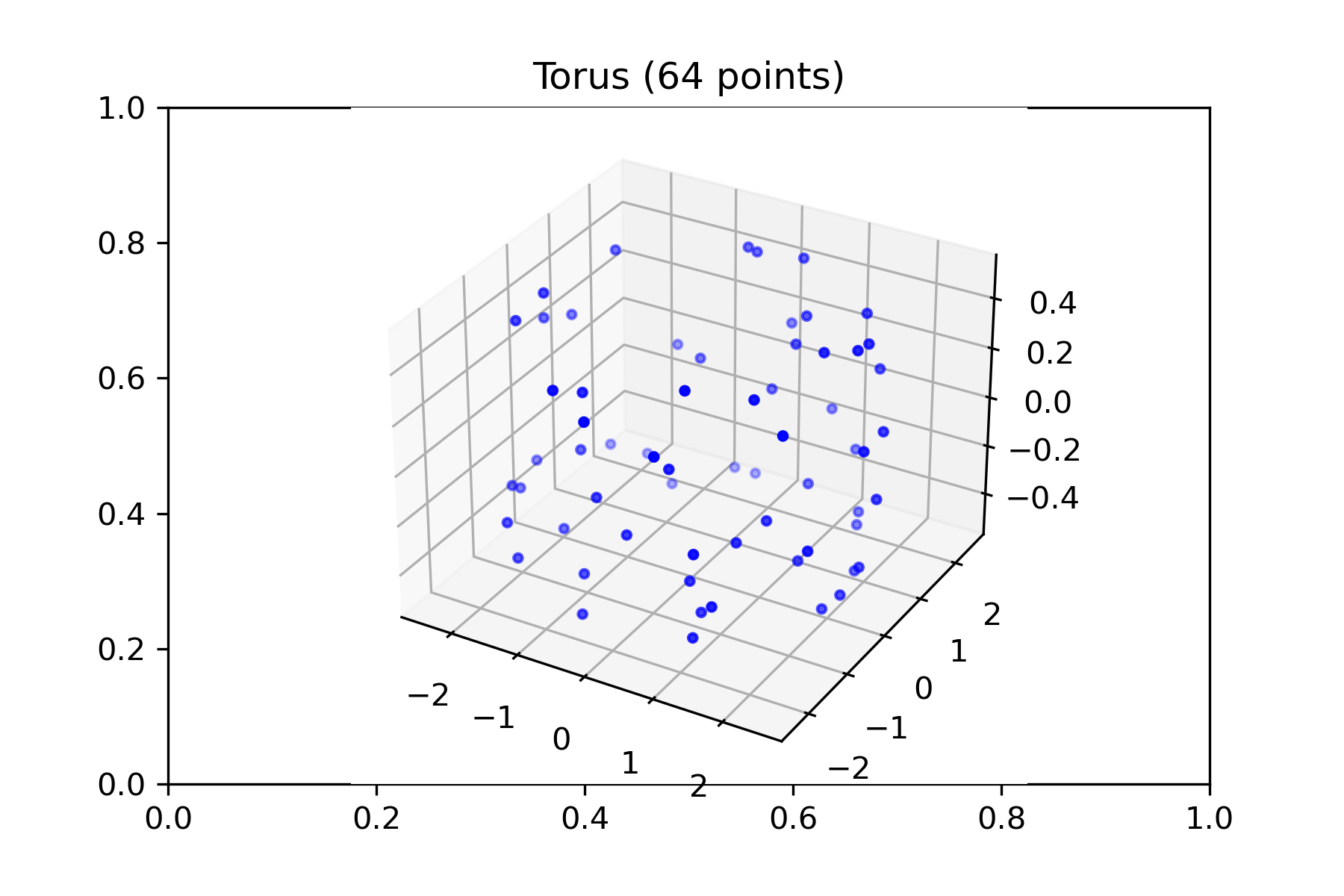}
    \includegraphics[width=0.4\textwidth]{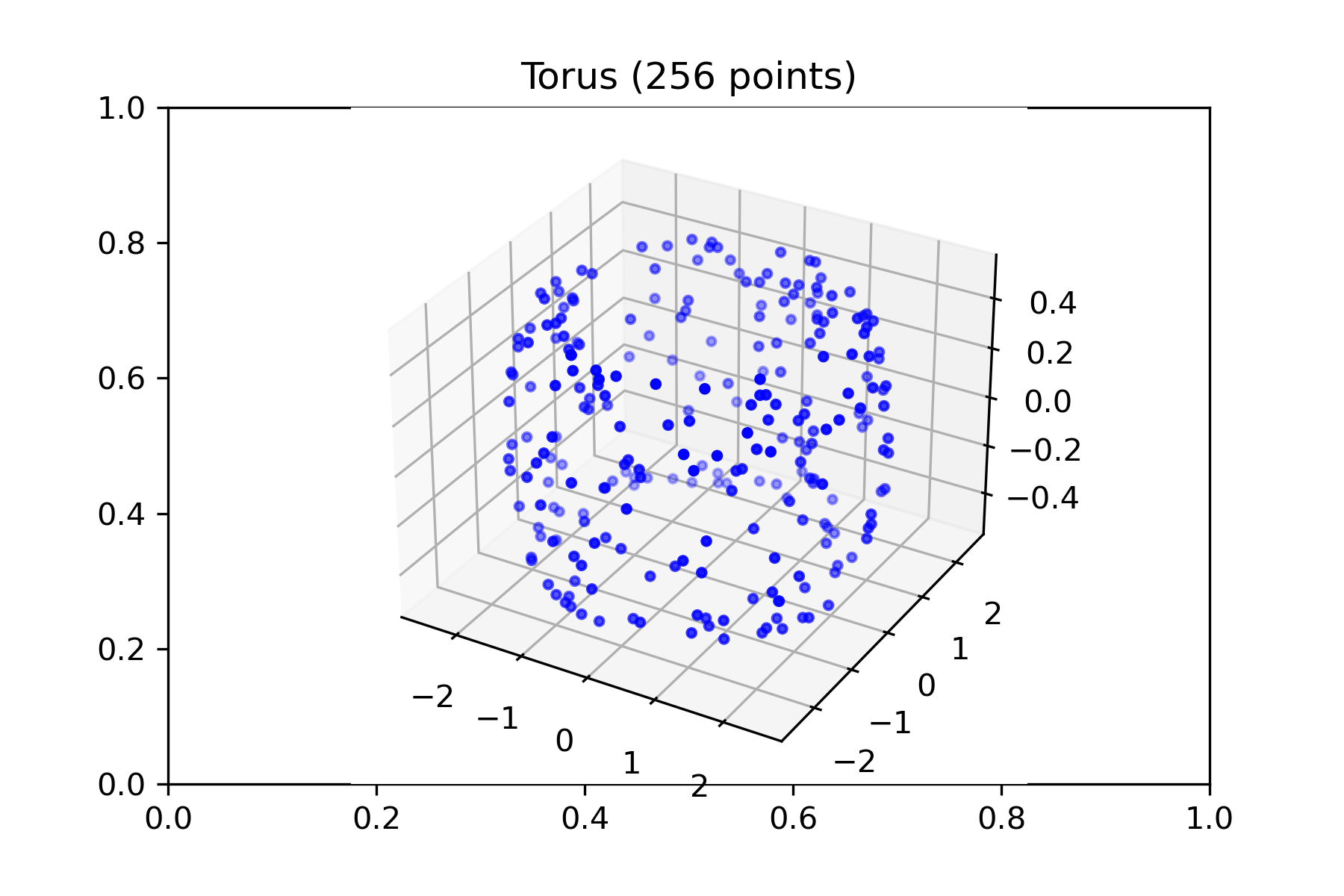}
    \includegraphics[width=0.4\textwidth]{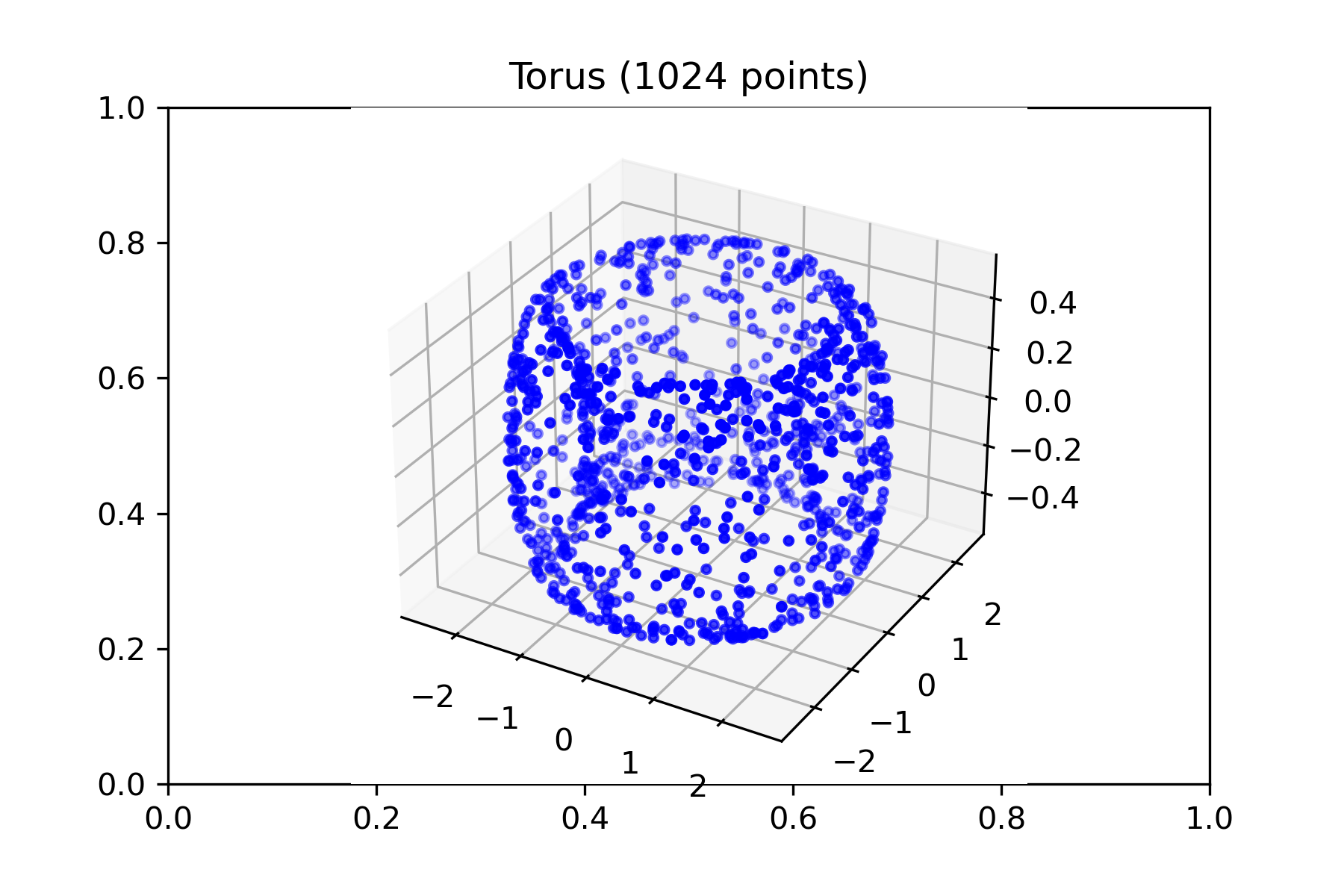}
    \caption{The samples obtained with a probability
      distribution proportional to the surface area.}
    \label{F:torus}
\end{figure}
Table \ref{T:torusF} and Figure \ref{F:torusF} show the corresponding
results obtained using the Fibonacci lattice method.
{\small
\begin{table}[ht]
\centering
\begin{tabular}{|c|c|c|c|c|}
\hline
\textbf{$m$} & \textbf{$L^\infty$ (Gauss)} & \textbf{$L^2$ (Gauss)} & \textbf{$L^\infty$ (Laplace)} & \textbf{$L^2$ (Laplace)} \\
\hline
64   & $2.75{\times}10^{-1}$ / $2.39{\times}10^{-1}$ & $1.86{\times}10^{-1}$ / $9.61{\times}10^{-2}$ & $1.24{\times}10^{-1}$ / $1.08{\times}10^{-1}$ & $8.90{\times}10^{-2}$ / $5.06{\times}10^{-2}$ \\
128  & $1.01{\times}10^{-2}$ / $9.62{\times}10^{-3}$ & $1.48{\times}10^{-2}$ / $6.71{\times}10^{-3}$ & $4.42{\times}10^{-2}$ / $3.27{\times}10^{-2}$ & $2.53{\times}10^{-2}$ / $1.18{\times}10^{-2}$ \\
256  & $3.39{\times}10^{-4}$ / $2.36{\times}10^{-4}$ & $1.66{\times}10^{-4}$ / $6.53{\times}10^{-5}$ & $3.00{\times}10^{-3}$ / $2.08{\times}10^{-3}$ & $5.84{\times}10^{-4}$ / $2.64{\times}10^{-4}$ \\
512  & $3.35{\times}10^{-7}$ / $2.31{\times}10^{-7}$ & $2.64{\times}10^{-7}$ / $8.60{\times}10^{-8}$ & $3.11{\times}10^{-4}$ / $2.08{\times}10^{-4}$ & $3.34{\times}10^{-5}$ / $1.64{\times}10^{-5}$ \\
1024 & $9.93{\times}10^{-10}$ / $8.59{\times}10^{-10}$ & $1.28{\times}10^{-8}$ / $6.10{\times}10^{-9}$ & $1.16{\times}10^{-6}$ / $7.40{\times}10^{-7}$ & $4.79{\times}10^{-7}$ / $1.85{\times}10^{-7}$ \\
\hline
\end{tabular}
\caption{Fibonacci lattice on the torus. Relative $L^\infty$ and $L^2$ errors for Gauss and Laplace kernels at different values of $m$. Each entry shows the error values for the two principal curvatures ordered according to size.}
\label{T:torusF}
\end{table}}
\begin{figure}[ht]
    \centering
    \includegraphics[width=0.4\textwidth]{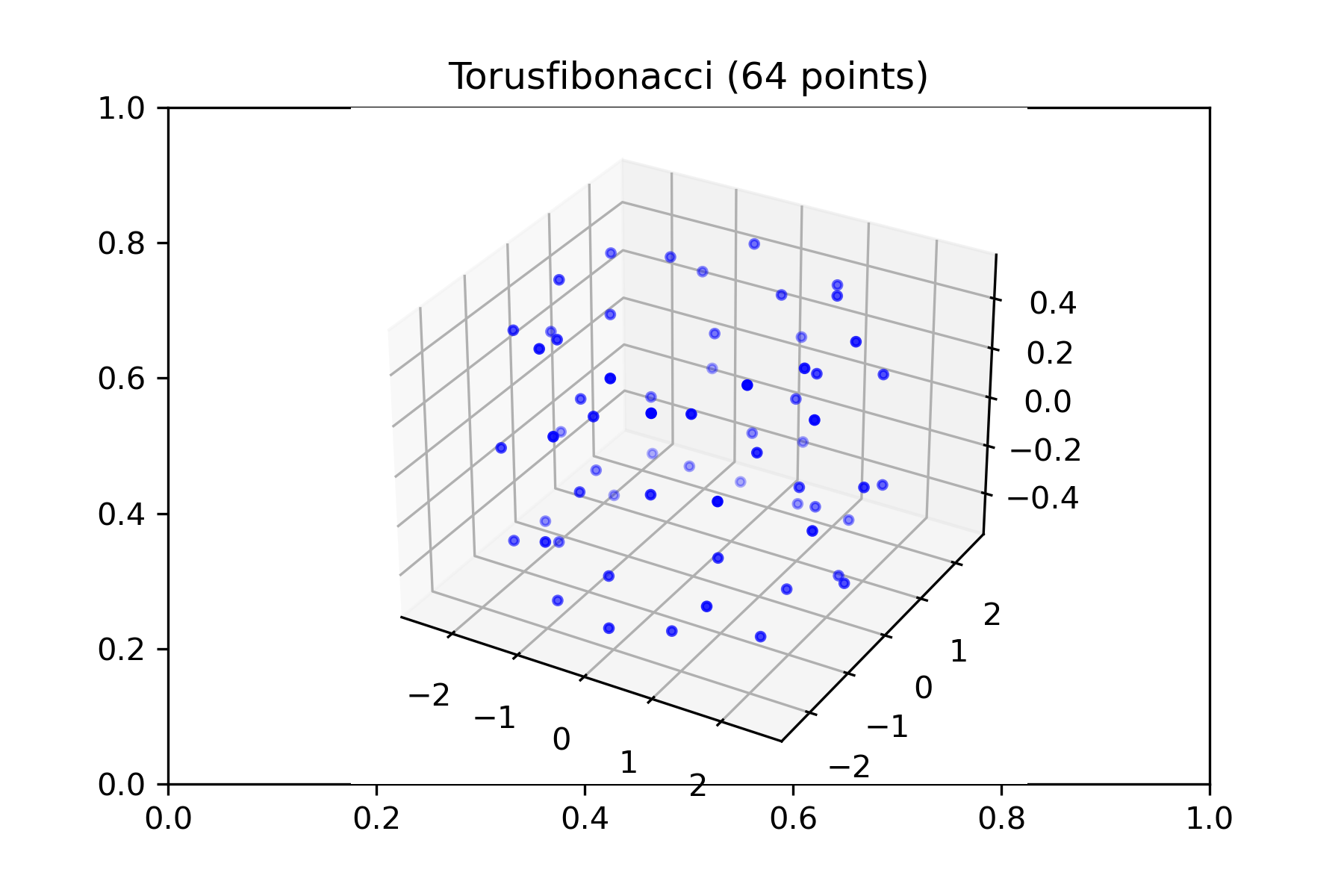}
    \includegraphics[width=0.4\textwidth]{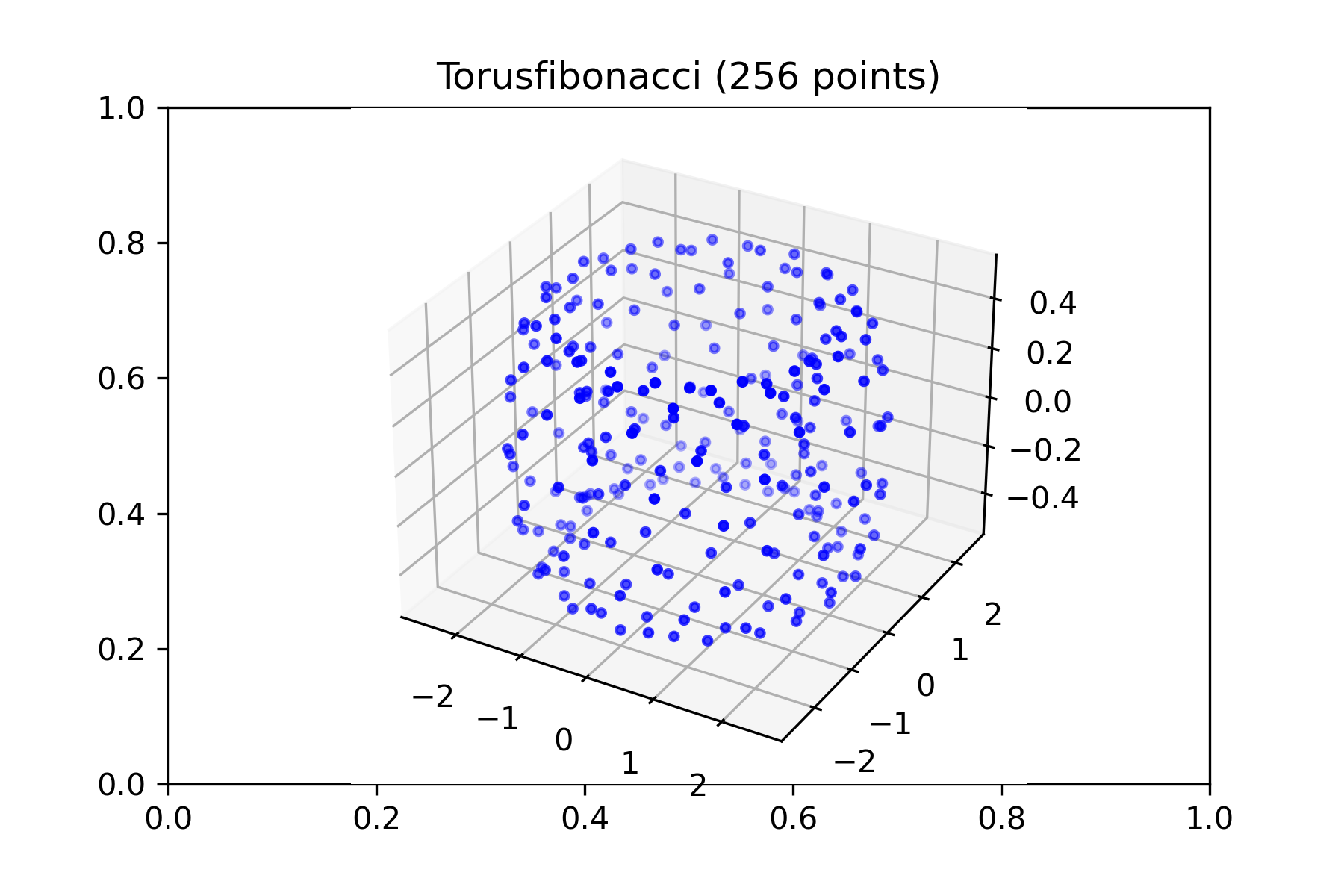}
    \includegraphics[width=0.4\textwidth]{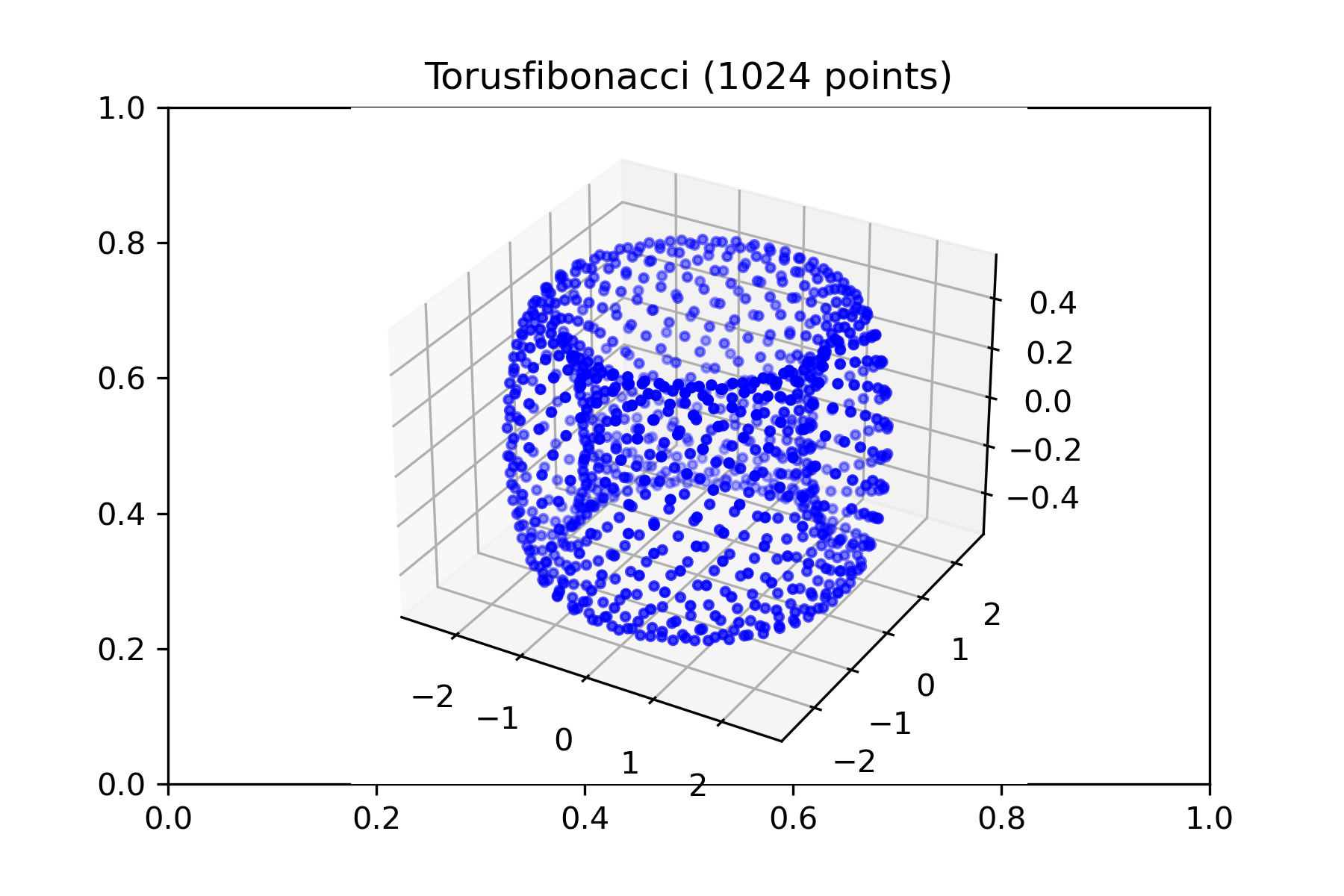}
    \caption{The Fibonacci samples used.}
    \label{F:torusF}
  \end{figure}
\subsection{The Gauss Curvature of an Ellipsoid}
Consider the ellipsoid
$[\frac{x_1^2}{a^2}+\frac{x_2^2}{b^2}+\frac{x_3^2}{c^2}=1]$  with
$a=2$, $b=.5$, $c=1$ and a Fibonacci discretization $\mathbb{X}$
consisting of $m$ points. Use the latter to approximate its Gauss
curvature
$$
K=\kappa_1\kappa_2=\frac{1}{a^2b^2c^2\bigl(\frac{x_1^2}{a^4}+
  \frac{x_2^2}{b^4}+\frac{x_3^2}{c^4}\bigr)^2}
$$
at 32 randomly chosen points (we first choose points uniformly on the sphere
using a normal distribution in the ambient space and normalize, then
scale their components by $a,b$, and $c$, respectively). Table
\ref{T:gaussK} record the relative error observed when using the Laplace kernel
with $\varepsilon=1$ and different discretization levels.
\begin{table}[h!]
\centering
\begin{tabular}{|c|c|c|c|c|c|}
\hline
$|\mathbb{X}|$ & 64 & 128 & 256 & 512 & 1024 \\
\hline
Relative error& $2.428 \times 10^{-1}$ & $8.408 \times 10^{-2}$ &
   $1.186 \times 10^{-2}$ & $5.439 \times 10^{-4}$ & $4.995 \times 10^{-5}$ \\
\hline
\end{tabular}
\caption{Gauss curvature error for an ellipsoid.}
\label{T:gaussK}
\end{table}
\subsection{The Noisy Case}
The framework developed allows one to seemlessly replace interpolation
with approximate interpolation by simply replacing the hard constraint
by a fidelity term. It is interesting that, in this context, this
regularization is also reflected in the level sets of the signature
function and hence has a geometric effect on the level
sets. We illustrate this with a few numerical experiments. We consider
curves again here since they are easier to visualize. In Figure
\ref{F:ellipse} we depict the reconstruction (in red) of an ellipse
based on an equispaced (in parameter domain) sample of 32 points which
are moved in a uniformly chosen random direction by a uniformly random
distance in $[0,0.1]$ and give rise to $\mathbb{X}$. The red curve
corresponds to the level $\overline{u}_\mathbb{X}$ given by
$$
\overline{u}_\mathbb{X}=\frac{1}{|\mathbb{X}|}\sum_{x\in
  \mathbb{X}}u_{\mathbb{X}}(x),
$$
i.e. the average value of the signature function on the data
set (which would be 1 for exact interpolation). Notice that we used
the regularization parameter $\alpha=.1$ to
counteract the effect of noise. We also visualize the implied normals
as well as the
implied sublevel sets as explained at the beginning of the section. In
addition we also draw in yellow the implied level 1 of the
exact interpolant computed on the unperturbed data set (no noise) to provide an
idea of the quality of the reconstruction.
\begin{figure}[ht]
    \centering
    \includegraphics[width=0.45\textwidth]{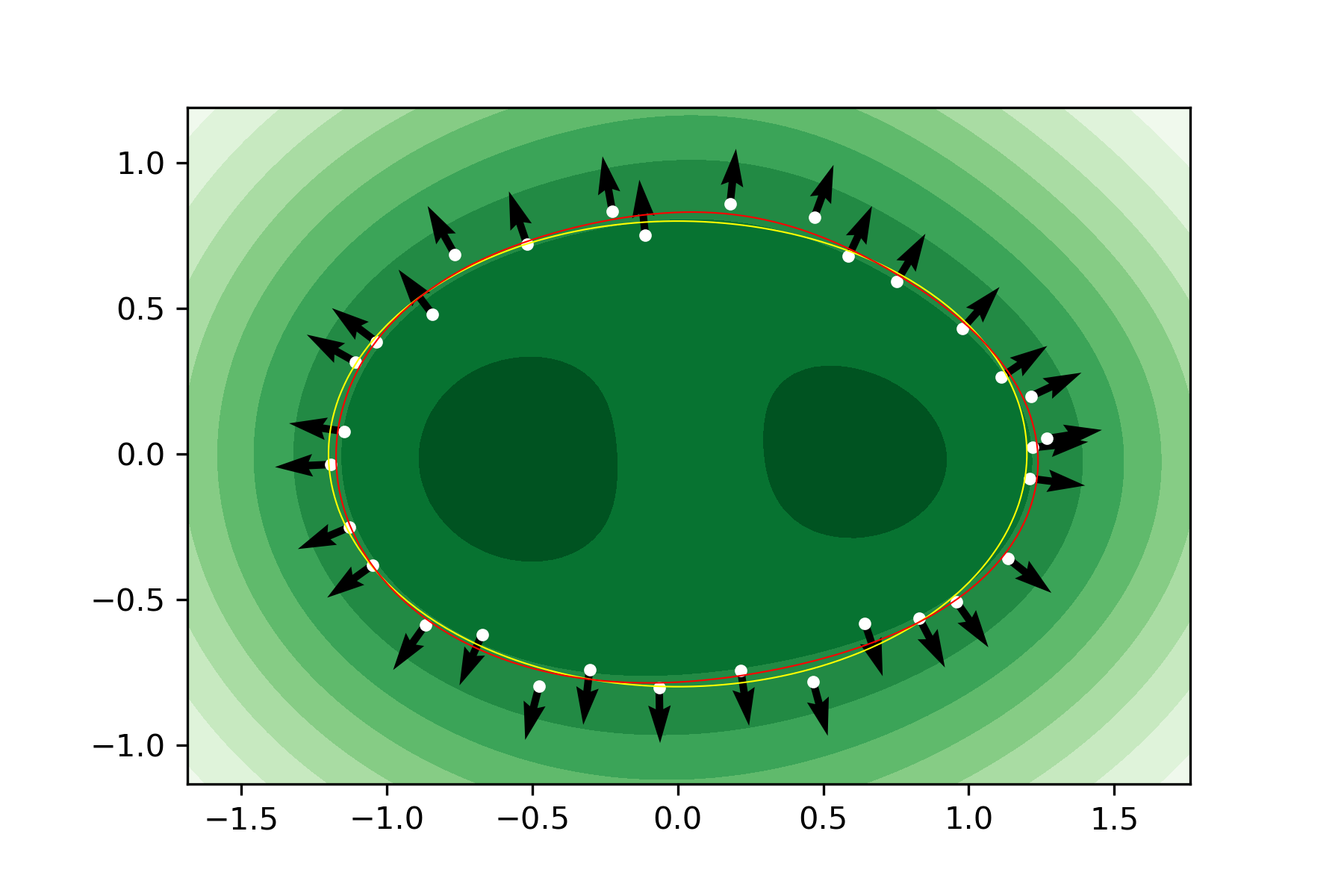}
    \includegraphics[width=0.45\textwidth]{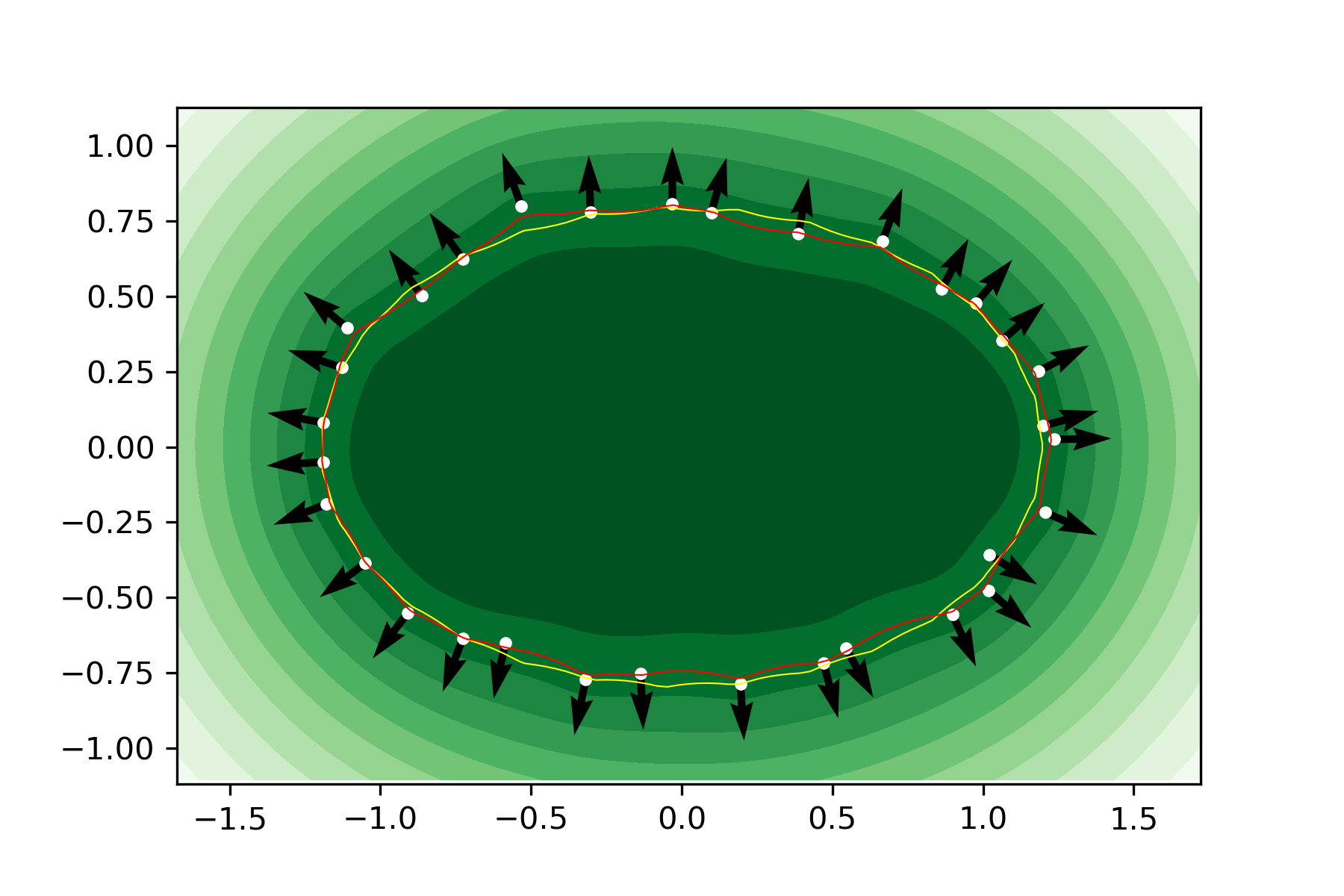}
    \caption{Reconstruction of an ellipse from a randomly perturbed
      sample of its points. The Gauss kernel is used on the left,
      while on the right the Laplace kernel is used with $\varepsilon=10^{-5}$.}
    \label{F:ellipse}
  \end{figure}
Notice how the Laplace kernel (with little regularization
$\varepsilon=10^{-5}$) exhibits high curvature at points in $\mathbb{X}$ and
in the original data set (not directly shown). If the Laplace kernel
is regularized more ($\varepsilon=1$), then it performs similarly to the
Gauss kernel as depicted in Figure \ref{F:ellipseSmooth} (left).
\begin{figure}[ht]
    \centering
    \includegraphics[width=0.45\textwidth]{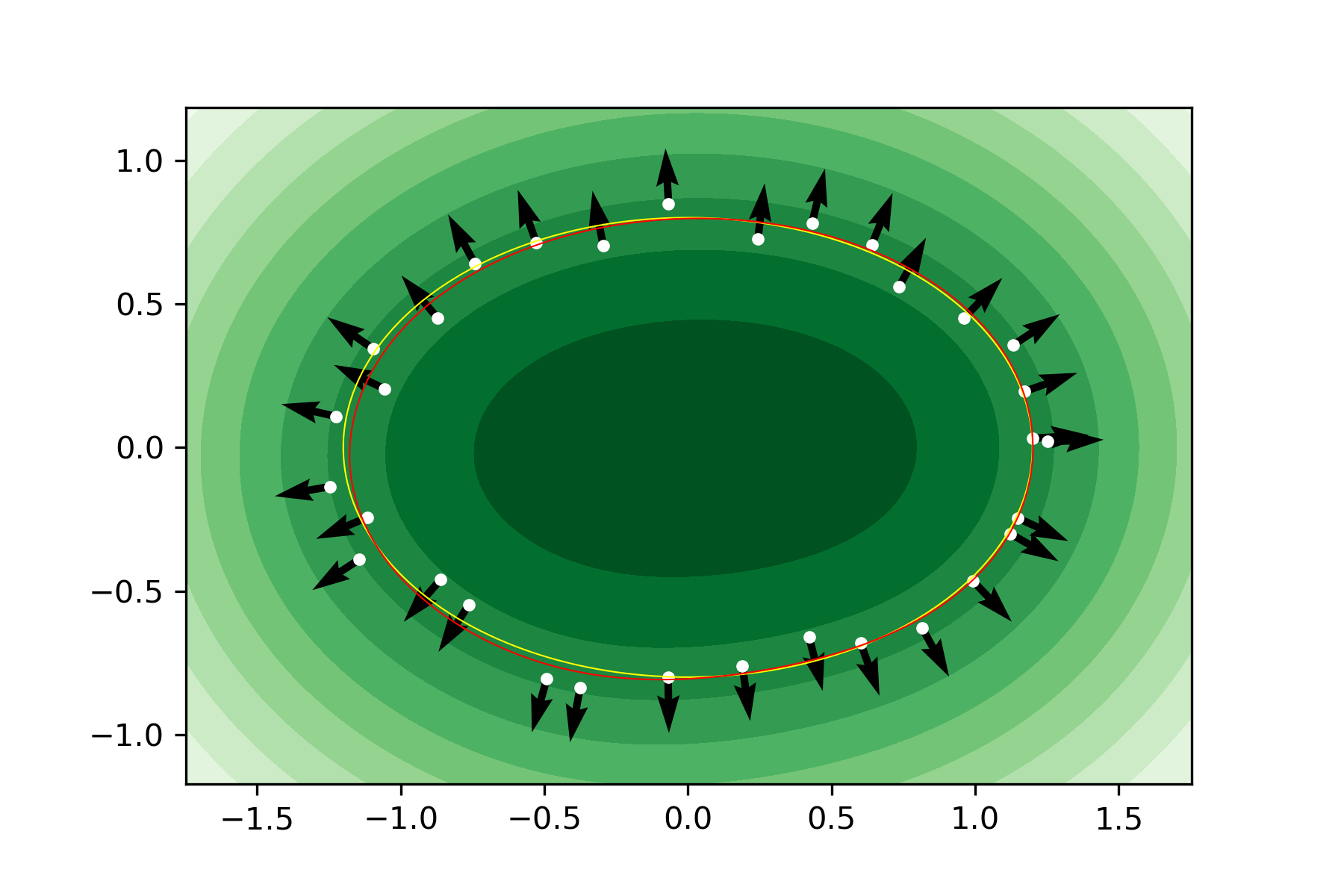}
    \includegraphics[width=0.45\textwidth]{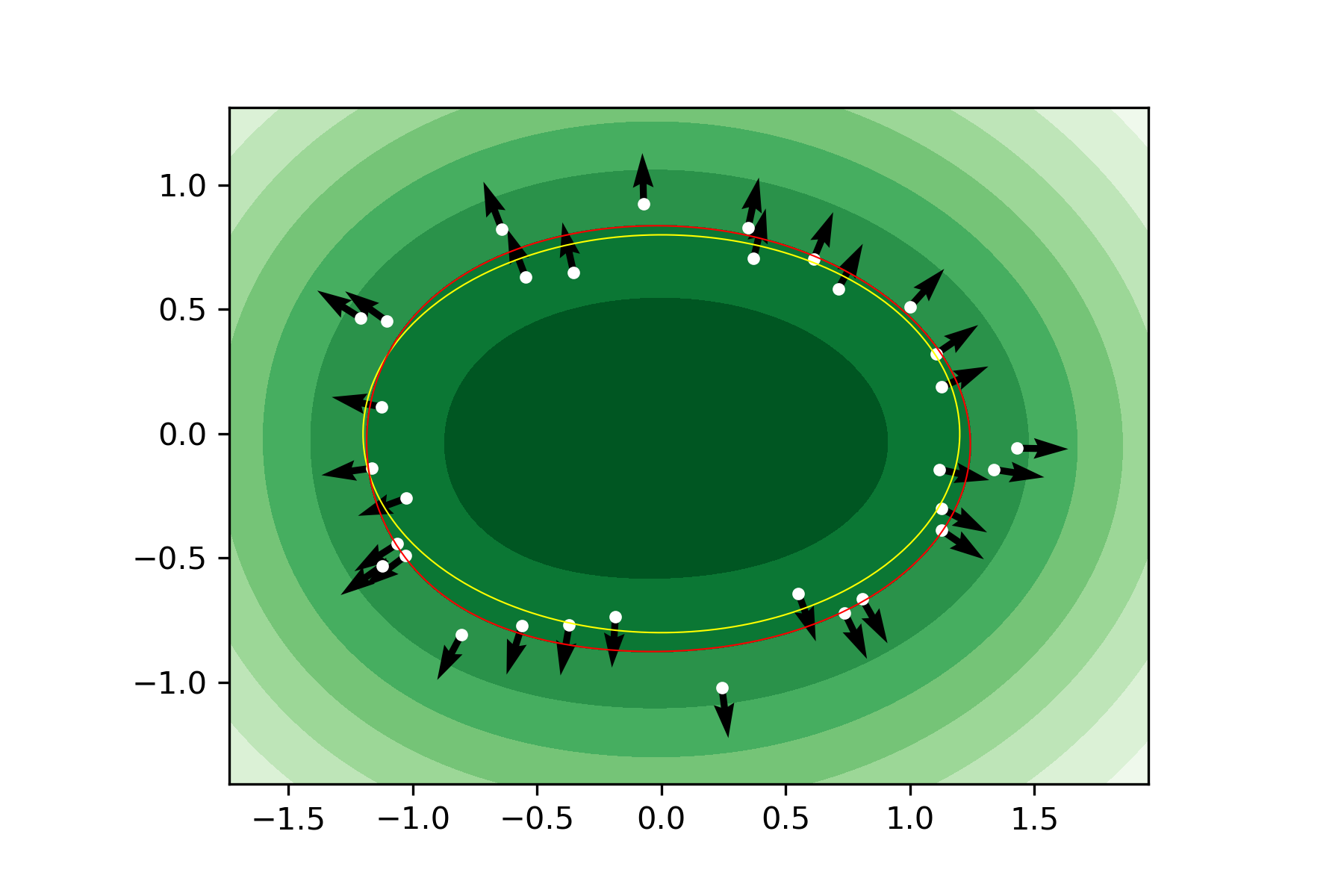}
    \caption{Reconstruction of an ellipse from a randomly perturbed
      sample of its points. The Laplace kernel is used with $\varepsilon=1$.}
    \label{F:ellipseSmooth}
  \end{figure}
The right picture in Figure \ref{F:ellipseSmooth} shows the
reconstruction obtained with the Laplace kernel ($\varepsilon=1$) when
each orginal datum is independently moved in a uniformly distributed
random direction to a random distance uniformly chosen in $[0,0.25]$.
\begin{rem}
  Finally we remark that, in some circumstances, it may be
  advantageous to use different kernels; one to obtain geometric
  information about $\mathbb{X}$ and one to interpolate the values
  $\mathbb{Y} $. This could be the case when dealing with functions of
  limited regularity defined on a smooth manifold.
\end{rem}
\bibliography{lite.bib} 
\end{document}